\documentclass[letterpaper]{article}

\usepackage[onehalfspacing]{setspace}
\usepackage[margin=2.5cm]{geometry}

\usepackage[utf8]{inputenc}
\usepackage[T1]{fontenc}
\usepackage{amssymb}
\usepackage{amsmath}
\usepackage{amsthm}
\usepackage{array}
\usepackage{tabularx}
\usepackage{float}
\usepackage{caption}
\usepackage{graphicx}
\usepackage{subcaption}
\usepackage{printlen}
\usepackage{booktabs}
\usepackage{longtable}
\usepackage{siunitx}
\usepackage[linesnumbered,ruled,vlined]{algorithm2e}
\usepackage{listings}
\usepackage{xcolor}
\usepackage{csquotes}
\usepackage{orcidlink}
\usepackage{authblk}

\usepackage{natbib}
\usepackage[short, nocomma]{optidef}

\usepackage{hyperref}
\usepackage{cleveref}

\newtheorem{theorem}{Theorem}[section]

\newtheorem{proposition}[theorem]{Proposition}
\newtheorem{corollary}[theorem]{Corollary}
\newtheorem{lemma}[theorem]{Lemma}
\newtheorem{definition}[theorem]{Definition}

\newtheorem{observation}[theorem]{Observation}

\begin{document}

\title{Revisiting transportation problems under Monge costs with applications to the Discrete Ordered Median Problem}

\author[1]{Stefan Nickel \orcidlink{0000-0002-8339-0117}}
\author[2,3]{Justo Puerto \orcidlink{0000-0003-4079-8419}}
\author[1]{Simon Ramoser \orcidlink{0000-0003-0318-4725}\thanks{Corresponding author: simon.ramoser@kit.edu}}
\author[2,3]{Alberto Torrejon \orcidlink{0000-0001-5156-1551}}

\affil[1]{Institute for Operations Research, Karlsruhe Institute of Technology, Karlsruhe, Germany}
\affil[2]{Department of Statistics and Operational Research, University of Seville, Sevilla, Spain}
\affil[3]{Institute of Mathematics, University of Seville, Sevilla, Spain}

\date{}

\maketitle

\begin{abstract}
We investigate the classical Hitchcock-Koopmans Transportation Problem under the assumption that the cost matrix satisfies the Monge property. In this setting, we derive compact formulas for optimal dual solutions based on the northwest-corner rule. As an application illustrating how these formulas yield structural insight while enhancing computational performance, we consider a broad class of facility location problems. In particular, the expressions are used within a Benders decomposition framework to derive novel mixed-integer linear programming formulations for the Discrete Ordered Median Problem with non-increasing weights. Numerical experiments validate that the resulting formulations achieve state-of-the-art performance and exhibit strong robustness across a wide range of instances.
\end{abstract}

\vspace{0.5em}

\section{Introduction}\label{sec:introduction}
The task of finding a cost-minimal transportation plan for carrying a commodity from a set of sources with given supply quantities to a set of destinations with given demands -- commonly referred to as the Hitchcock-Koopmans Transportation Problem (TP) -- is a classical and well-studied problem in the field of Operations Research. It has numerous applications beyond the optimal distribution of goods, such as production scheduling, supply chain management, or emergency logistics, just to name a few \citep[see, e.g.,][]{zhang_review_2025}.

An interesting special case of the TP arises when its cost matrix satisfies the \textit{Monge property}, a notion closely connected to submodular functions \citep[see][]{queyranne_general_1998}. Formally, a $p \times q$ matrix $c$ is a Monge matrix if $c_{ij} + c_{i'j'} \leq c_{ij'} + c_{i'j}$ holds for all $1 \leq i < i' \leq p$ and $1 \leq j < j' \leq q$. Intuitively, this inequality allows any two crossing transportation paths to be replaced by non-crossing ones without increasing the total cost, thereby ensuring the existence of an optimal transportation plan where no two paths cross. Under this assumption, the TP can be solved by a greedy algorithm in linear time. The Monge property and related structural conditions have also been exploited in other optimization problems, enabling solution algorithms that are unavailable for arbitrary cost matrices. For instance, Monge-like structures have been used to derive polynomial-time algorithms for problems that are NP-hard in general, including the Traveling Salesman Problem and the Quadratic Assignment Problem \citep[see][]{BurkardKlinzRudolf1996, CELA2018}. In the context of the TP, Monge costs emerge, among others, in geometric and temporal settings where sources and destinations admit a natural sequential ordering. For example, consider $p$ sources and $q$ destinations indexed by integers $i = 1, \dots, p$ and $j = 1, \dots, q$. If the pairwise costs are defined by the one-dimensional Euclidean distance $c_{ij} := |i-j|$, the resulting matrix $c$ satisfies the Monge property. A prominent application of this is found in computer vision for measuring image similarity by comparing grayscale histograms \citep{WERMAN1985, RUBNER1998}. Another geometric setup that induces Monge costs occurs when sources are distributed along a straight line and destinations are distributed along a parallel line. If costs again correspond to Euclidean distances between sorted source and destination nodes, the resulting cost matrix is Monge. There are, for instance, power line connection problems which can be modeled as a TP with this cost structure \citep{AGGARWAL1995}. In temporal settings, the natural sequential ordering is dictated by time. An example is the job sequencing problem described by \citet{Hoffman1963SimpleLP}, where the goal is to minimize the maximum tardiness of all jobs. Formulating this as a TP that assigns processing time units to jobs ordered by their due dates gives rise to a Monge cost matrix. More recent applications of the TP in conjunction with Monge cost structures can be found in inventory allocation, where suitable assumptions related to product substitution costs allow for optimal greedy distribution of inventory \citep{TSEYTLIN2019, TSEYTLIN2023}.

Despite being useful as a standalone model, in many modern applications the TP appears as a subproblem that must be solved repeatedly and efficiently. This occurs, for instance, within Benders decompositions \citep[e.g.,][]{fischetti_benders_2016}, bilevel formulations \citep[e.g.,][]{hillbrecht2025bilevel, singh_bilevel_TP}, or Farkas‑based valid inequality schemes \citep[e.g.,][]{espejo_single_alloc_hub}. In such settings, Monge costs offer an advantage that goes beyond the acceleration of solution methods: they can reveal structural insights about the overarching models and thereby enable a more efficient use of advanced optimization techniques. Our approach is to derive compact formulas for optimal dual variables of the TP under Monge costs that can be fruitfully used to gain such insights, specifically to derive novel explicit problem reformulations.

To make our ideas concrete, we conduct a case study in the field of location analysis, where many models involve a step that sorts costs before weighting them, for example, in optimal subtree location on trees \citep[e.g.,][]{PuertoTamir2005, Pozo2024} and in hub location \citep[e.g.,][]{PUERTO2011, Pozo2021}. 
The sorting step imposes a sequential structure on the allocation costs, making these models a natural place in which Monge matrices can appear. Particularly, we consider the Discrete Ordered Median Problem (DOMP), which generalizes a large variety of location problems. Given a number $n$ of potential facility locations and a number $p \leq n$ of facilities that are to be opened, the goal is to choose the $p$ best facility locations in the sense that weighted delivery costs from these locations to customer locations are minimized. The weights are order-dependent, meaning that they are applied to the non-decreasingly sorted delivery costs. By appropriate choices of the weight vector, the DOMP encompasses many classical location models as special cases, including the $p$-median, $p$-center, and $k$-centrum problems. In this study, we consider instances with non‑increasing order‑dependent weights. Such objectives occur, for instance, in obnoxious facility problems, where \enquote{undesired} facilities, such as wind turbines, power plants or landfills, must be located \citep{church_review_2022}. Motivated by recent advances in this area \citep[e.g.,][]{bigler_matheuristic_2024, kalczynski_obnoxious_2026}, we apply Benders decomposition to a mixed-integer linear programming (MILP) formulation for the DOMP proposed by \citet{labbe_comparative_2017}. The Benders subproblem for generating optimality cuts is a TP whose cost matrix becomes Monge under the assumption of non-increasing weights. Because these cuts are derived from optimal dual solutions to this TP, our formulas yield explicit representations of the Benders cuts, leading to new models.

\subsection{Literature review}
\subsubsection{The Transportation Problem}
Although the history of the TP dates back to the 18th century when Gaspard Monge was concerned with transporting earth \citep{monge1781memoire}, the problem's greatest developments occurred in the 20th century. While \citet{Kantorovich1942} generalized the problem, thereby grounding the theory of optimal transport, \citet{hitchcock_distribution_1941} and \citet{koopmans_optimum_1949} independently formalized the discrete version considered in this paper. Early major contributions to the latter include the application of the simplex method by \citet{dantzig1951transportation} and the network flow algorithm developed by \citet{ford_solving_1956}. Since then, researchers from a wide range of disciplines have worked on the TP, resulting in a vast body of literature. We therefore limit this brief historical overview to the above mentioned milestones and refer to \citet{zhang_review_2025} for a recent review of the problem with modern perspectives. A broad overview of extensions of the TP, together with their corresponding developments and mathematical formulations, is presented by \citet{kacher_comprehensive_2021}. Structural insights and polyhedral foundations related to the TP are treated in the classical textbook by \citet{yemelichev1984polytopes}. Finally, for optimization problems under Monge cost structures, the focus of this paper, the standard references remain the extensive surveys by \citet{BurkardKlinzRudolf1996} and \citet{burkard_monge_2007}, which include a discussion of the TP. In this context, a seminal work is due to \citet{Hoffman1963SimpleLP}, who showed that the TP can be solved greedily to optimality when the costs satisfy a certain structural condition. Since the underlying idea of this condition was already contained in Monge's 1781 work, Hoffman named this property after Monge.

\subsubsection{The Discrete Ordered Median Problem}
At the beginning of the century, \citet{nickel_discrete_2001} laid the foundations of the DOMP literature by formulating the problem as a mixed-integer quadratic program and subsequently linearizing it to obtain the first MILP model for the DOMP. The early developments on the problem are comprehensively summarized in the textbook by \citet{nickel_puerto_2005}, including improved models as well as exact and heuristic solution algorithms by \citet{boland_exact_2006, DomnguezMarin2003}; and \citet{dominguez-marin_heuristic_2005}. In 2009, \citet{marin_flexible_2009} proposed a DOMP model capable of solving specific instances with up to $n = 100$ potential facility locations within reasonable time.  A new modeling approach based on rewriting the ordered median objective function of the DOMP in terms of $k$-sums and exploiting known theory from $k$-sum optimization \citep{nickel_velten_flexible_2017, puerto_revisiting_2017} gave rise to further improvements. Particularly, this idea has led to DOMP formulations that are capable of solving a wide range of instances with up to $n = 200$ locations \citep{marin_fresh_2020}. More advanced decomposition-based optimization techniques such as Branch-and-Price \citep{deleplanque_branch-price-and-cut_2020}, Benders decomposition \citep{ljubic_benders_2024}, or customized Branch-and-Bound with special DOMP subproblems \citep{cherkesly_ranking_2025} were applied successfully, resulting in state-of-the-art model performance.

It is well-known that the DOMP contains obnoxious facility location problems as special cases, including the obnoxious $p$-median and $p$-center problems. Within the DOMP literature, such instances are recognized as particularly challenging, largely because negative objective weights substantially affect the problem structure. Consequently, recent computational studies employ obnoxious criteria both to demonstrate the strength of their formulations and to find their computational limits \citep{ljubic_benders_2024, cherkesly_ranking_2025}. The obnoxious $p$-median and $p$-center problems have also been studied independently of the DOMP. In recent years, new models and solution approaches have been proposed \citep[see, e.g.,][]{chiang_obnox_2017, colmenar_obnox_grasp_2016, herran_obnox_vns_2020, bigler_matheuristic_2024, kalczynski_obnoxious_2026}. For a broader overview of obnoxious facility location problems, we refer to the survey by \citet{church_review_2022}.

\subsection{Contributions}
With this work, we contribute to the literature as follows.
\begin{itemize}
    \item We explicitly derive compact expressions for optimal dual variables of TPs under Monge costs through a close analysis of the primal solution generated by the greedy northwest-corner algorithm.
    \item We show how the results on TPs under Monge costs can be exploited when such problems appear as substructures within advanced optimization frameworks by deriving novel MILP formulations for the DOMP with non-increasing objective weights.
    \item We conduct an extensive computational study comparing the proposed formulations with state-of-the-art DOMP models on a broad range of instances, including prominent obnoxious facility location problems, and demonstrate that the proposed models achieve highly competitive, and in several cases superior, computational performance.
\end{itemize}

The remainder of this paper is structured as follows. Section~\ref{sec:transportation} starts with a detailed presentation of the TP under Monge costs, and recalls the greedy northwest-corner algorithm together with its extension for retrieving a dual solution. Building on this, it is shown how a so-obtained dual solution can be expressed compactly using formulas that encode the primal solution generated by the northwest-corner algorithm. In Section~\ref{sec:application_to_domp}, the DOMP is introduced formally along with a MILP model from the literature to which a Benders reformulation is applied. The dual solutions identified in Section~\ref{sec:transportation} give rise to two new exponential MILP formulations for DOMP with non-increasing weights; we conclude the section with several theoretical and practical remarks. Section~\ref{sec:numerical_experiments} reports computational experiments that compare the new formulations with two state‑of‑the‑art DOMP models. Section~\ref{sec:conclusion} ends the paper with a concise summary and suggestions for further research.

\section{The Transportation Problem under Monge costs}\label{sec:transportation}
This section develops our core results regarding the TP under the assumption of Monge costs. In Subsection~\ref{sec:TP_definition}, we define the TP under Monge costs and recall the well-known northwest-corner method that finds optimal primal solutions to the studied problem. In addition, we prove a characterization of these solutions that is solely based on the given supply and demand quantities (Proposition~\ref{NW_corner_characterization_2}). In Subsection~\ref{sec:TP_formulas}, we proceed to our main results of this section: compact expressions for solutions to the dual TP under Monge costs (Theorems~\ref{theorem_u_v_formulas} and~\ref{theorem_u_v_formulas_alternative}). This is achieved by analyzing a known extension of the northwest-corner algorithm that recovers an optimal dual solution encoding the corresponding primal solution. To our knowledge, these explicit compact expressions have not appeared in the literature before.

\subsection{Problem definition and the northwest-corner solution method}\label{sec:TP_definition}
Let $p \in \mathbb{Z}_{\geq 1}$ be the number of supply nodes with supplies $s \in \mathbb{R}_{\geq 0}^p$, and $q \in \mathbb{Z}_{\geq 1}$ be the number of demand nodes with demands $d \in \mathbb{R}_{\geq 0}^q$. Moreover, let $c \in \mathbb{R}_{\geq 0}^{p \times q}$ be the transportation cost matrix, where entry $c_{ij}$ prescribes the costs of carrying one unit of the transported commodity from supply node $i$ to demand node $j$. The goal is to find the non-negative quantities that should be transferred between any pair of supply and demand nodes such that all supplies and demands are exactly satisfied and the total transportation cost is minimal. It is easy to see that this is possible if and only if supplies and demands are \enquote{balanced}, so we assume $\sum_{i = 1}^p s_i = \sum_{j = 1}^q d_j$.

Consider the standard LP formulation of the classical (balanced) TP, first given by \citet{hitchcock_distribution_1941}, and stated here as model~\eqref{transportation}. It employs continuous, non-negative decision variables $x_{ij}$ for all $(i, j) \in [p] \times [q]$\footnote{For $n \in \mathbb{Z}_{\geq 1}$, we use the notation $[n] = \{1, \dots, n\}$.}, representing the quantity shipped from supply node $i$ to demand node $j$.
\begin{mini!}[2]<b>
    {}
    {\sum_{i = 1}^p \sum_{j = 1}^q c_{ij} \, x_{ij}}
    {\label{transportation}}{\text{TP}:\,}
    \addConstraint{\sum_{j = 1}^q x_{ij}}{= s_i \phantom{,}\quad}{\forall \, i \in [p] \label{transportation:c1}}
    \addConstraint{\sum_{i = 1}^p x_{ij}}{= d_j \phantom{,}\quad}{\forall \, j \in [q] \label{transportation:c2}}
    \addConstraint{x_{ij}}{\geq 0 \phantom{,}\quad}{\forall \, i \in [p],\, j \in [q] \label{transportation:c3}}
\end{mini!}

Feasibility of this LP is guaranteed by the balancedness condition, and the problem is clearly bounded. Moreover, since the constraint matrix is totally unimodular, any basic feasible solution will be integer-valued if the right-hand sides are integral \citep[see, e.g.,][]{bazaraa2010linear}.

Throughout this section, we assume that $c$ satisfies the Monge property. Recall from the introduction that this requires the inequality $c_{ij} + c_{i'j'} \leq c_{ij'} + c_{i'j}$ to hold for all $1 \leq i < i' \leq p$, $1 \leq j < j' \leq q$. In that case, the greedy northwest-corner rule always yields an optimal solution to \eqref{transportation}. This algorithm starts in the northwest-corner of the $p \times q$ variable matrix, that is, at index pair $(i, j) = (1, 1)$ and sets $x_{ij}$ to the largest number of units possible to be transported from supplier $i$ to demand point $j$. After updating the supplies and demands according to the already sent units, the procedure advances to cell $(i+1, j)$ if supplier $i$ has reached its capacity or to cell $(i, j+1)$ if demand $j$ is satisfied. Like this, the procedure iterates until all available units are shipped. Algorithm~\ref{NW_corner_rule} is the version of this greedy northwest-corner rule presented by \citet{queyranne_general_1998} and specialized to the classical TP.

\begin{algorithm}[ht]
\caption{Northwest-corner rule}\label{NW_corner_rule}
\DontPrintSemicolon
$x \gets \{\}$\;
$\eta \gets \{\}$\;
$\delta_i^1 \gets s_i \quad \forall \, i \in [p]$ \tcp*[r]{remaining supplies}
$\delta_j^2 \gets d_j \quad \forall \, j \in [q]$ \tcp*[r]{remaining demands}
$a \gets (1,1)$ \tcp*[r]{start in the northwest-corner}
$t \gets 0$\;
\While{$a_2 \leq q$}{
    $t \gets t+1$\;
    $a^t \gets a$\;
    $x_{a^t} \gets \min\{\delta_{a^t_1}^1, \delta_{a^t_2}^2\}$; append $(a^t, x_{a^t})$ to $x$\;
    $\delta_{a^t_1}^1 \gets \delta_{a^t_1}^1 - x_{a^t}$\;
    $\delta_{a^t_2}^2 \gets \delta_{a^t_2}^2 - x_{a^t}$\;
    \If{$\delta_{a^t_1}^1 = 0$ and $a^t_1 < p$}{
        $\eta_t \gets 1$; append $\eta_t$ to $\eta$\;
        $a^t_1 \gets a^t_1+1$ \tcp*[r]{move down}
    }
    \Else{
        $\eta_t \gets 2$ ; append $\eta_t$ to $\eta$\;
        $a^t_2 \gets a^t_2+1$ \tcp*[r]{move right}
    }
    $a \gets a^t$
}
$T \gets t$\;
\Return{$\{(a^1, x_{a^1}), \dots, (a^T, x_{a^T})\}, \{\eta_1, \dots, \eta_{T-1}\}$}
\end{algorithm}

Algorithm~\ref{NW_corner_rule} moves along a staircase-shaped path from cell $(1,1)$ to cell $(p,q)$: between cell $a^t$ and $a^{t+1}$, it either advances one row (\enquote{down}) or one column (\enquote{right}) on the $p \times q$ grid. This movement is well understood and can be characterized as follows.
\begin{observation}\label{NW_corner_characterization}
    We have $(i, j) \in \{ a^1, \dots, a^T \} =: \mathcal{A}$ if and only if one of the following holds:
    \begin{enumerate}
        \item[(i)] $(i, j) = (1, 1)$
        \item[(ii)] $(i-1, j) \in \mathcal{A}$ and $\sum_{k=1}^{i-1} s_k \leq \sum_{k=1}^{j} d_k$ and $i-1 < p$
        \item[(iii)] $(i, j-1) \in \mathcal{A}$ and $\bigl(\sum_{k=1}^{i} s_k > \sum_{k=1}^{j-1} d_k$ or $i = p \bigr)$
    \end{enumerate}
    \begin{proof}
        The algorithm is initialized with $a^1 = (1, 1)$. From there, it proceeds by moving one row downwards or one column to the right. \textit{(ii)} characterizes a downward movement from $(i-1, j)$ to $(i, j)$. This happens exactly when $i - 1 < p$ and the total supply up to $i-1$ does not exceed the total demand up to $j$, that is, when $\sum_{k=1}^{i-1} s_k \leq \sum_{k=1}^{j} d_k$. Similarly, \textit{(iii)} describes a move to the right from $(i, j-1)$ to $(i, j)$. Such a move occurs exactly when the total supply up to $i$ exceeds the total demand up to $j$, that is, $\sum_{k=1}^{i} s_k > \sum_{k=1}^{j-1} d_k$, or when $i = p$. Note that in the latter case, it holds $\sum_{k=1}^{i} s_k = \sum_{k=1}^{j-1} d_k$. Overall, we see that \textit{(i), (ii), (iii)} inductively describe the cells $a^1, \dots, a^T$ traversed by Algorithm~\ref{NW_corner_rule}.
    \end{proof}
\end{observation}

Staircase patterns in transportation matrices have been studied from a broader perspective by \citet{dahl_transportation_2008}. The pattern of a transportation matrix $(x_{ij})_{\substack{i \in [p] \\ j \in [q]}}$ refers to the positions of its nonzero entries. In his work, Dahl characterizes transportation matrices with a given staircase pattern in terms of certain \enquote{critical positions} within that pattern. In particular, he shows that if specific inequalities involving the input vectors $s$ and $d$ hold for each of these critical positions, then the path generated by the northwest-corner rule is contained in that pattern. Since we focus exclusively on the staircase pattern traced by the northwest-corner algorithm, we can provide, for our purposes, a more convenient characterization that does not rely on critical positions. This novel result is stated in Proposition~\ref{NW_corner_characterization_2}.

\begin{proposition}\label{NW_corner_characterization_2}
    For the cells $\mathcal{A} = \{a^1, \dots, a^T\}$ traversed by Algorithm~\ref{NW_corner_rule}, it holds
    \begin{equation}
        \mathcal{A} = \bigg\{(i,j) \in [p] \times [q] \, \bigg| \, \bigg( \sum_{k=1}^{i-1} s_k \leq \sum_{k=1}^{j} d_k \, \lor \, i = 1 \bigg) \land \bigg( \sum_{k=1}^{i} s_k > \sum_{k=1}^{j-1} d_k \, \lor \, i = p \, \lor \, j = 1 \bigg) \bigg\}. \label{NW_characterization_equation}
    \end{equation}
    \begin{proof}
        The key structural insight leading to this claim is that if a tuple $(i, j)$ lies strictly \enquote{below} the staircase path $\mathcal{A}$, $\sum_{k=1}^{i-1} s_k \leq \sum_{k=1}^{j} d_k$ \textit{cannot} hold true. Similarly, if $(i, j)$ lies strictly \enquote{above} the staircase, $\sum_{k=1}^{i} s_k > \sum_{k=1}^{j-1} d_k$ \textit{cannot} hold true. Intuitively, this means that the condition in Observation~\ref{NW_corner_characterization} \textit{(ii)} for down moves of Algorithm~\ref{NW_corner_rule} can never hold below the staircase while the condition in Observation~\ref{NW_corner_characterization} \textit{(iii)} for right moves cannot hold above the staircase. In that sense, the cells $\mathcal{A}$ are exactly at the intersection between feasible down moves and right moves of Algorithm~\ref{NW_corner_rule}.

        We start by proving the aforementioned structural insight. Let $(i, j) \in [p] \times [q]$ such that $i > a_1^t$ for every $t \in [T]$ with $a_2^t = j$, that is, $(i, j) \notin \mathcal{A}$ lies strictly below the staircase. Suppose for contradiction that $\sum_{k=1}^{i-1} s_k \leq \sum_{k=1}^{j} d_k$. The staircase-like movement of Algorithm~\ref{NW_corner_rule} implies the existence of $\tilde{i} < i$ such that $(\tilde{i}, j) \in \mathcal{A}$ and $(\tilde{i}, j+1) \in \mathcal{A}$. Note that it must hold $j < q$ because otherwise $(i, j) \in \mathcal{A}$. By Observation~\ref{NW_corner_characterization} \textit{(iii)}, we have $\sum_{k=1}^{\tilde{i}} s_k > \sum_{k=1}^{j} d_k$ or $\tilde{i} = p$. The latter is not possible since $\tilde{i} < i \leq p$, so
        \begin{equation*}
            \sum_{k=1}^{j} d_k < \sum_{k=1}^{\tilde{i}} s_k \leq \sum_{k=1}^{i-1} s_k \leq \sum_{k=1}^{j} d_k.
        \end{equation*}
        We conclude from this contradiction that the assumption was wrong.

        Similarly, let $(i, j) \in [p] \times [q]$ such that $j > a_2^t$ for every $t \in [T]$ with $a_1^t = i$, that is, $(i, j) \notin \mathcal{A}$ lies strictly above the staircase. Suppose for contradiction that $\sum_{k=1}^{i} s_k > \sum_{k=1}^{j-1} d_k$ or $i=p$. The latter is not possible since $(p, q) \in \mathcal{A}$ and thus, the assumption on $j$ would imply $j > q$. So, assume $i < p$ and $\sum_{k=1}^{i} s_k > \sum_{k=1}^{j-1} d_k$. The staircase-like movement of Algorithm~\ref{NW_corner_rule} implies the existence of $\tilde{j} < j$ such that $(i, \tilde{j}) \in \mathcal{A}$ and $(i+1, \tilde{j}) \in \mathcal{A}$. By Observation~\ref{NW_corner_characterization} \textit{(ii)}, we have $\sum_{k=1}^{i} s_k \leq \sum_{k=1}^{\tilde{j}} d_k$. Therefore,
        \begin{equation*}
            \sum_{k=1}^{j-1} d_k < \sum_{k=1}^{i} s_k \leq \sum_{k=1}^{\tilde{j}} d_k \leq \sum_{k=1}^{j-1} d_k.
        \end{equation*}
        This shows the second part of the key structural insight.

        For the remainder of this proof, the right-hand side of \eqref{NW_characterization_equation} is denoted by $\mathcal{B}$. We proceed by proving $\mathcal{A} \subseteq \mathcal{B}$. For this, let $(i, j) \in \mathcal{A}$. By Observation~\ref{NW_corner_characterization} three scenarios are possible. \textit{(i)} If $(i, j) = (1, 1)$, trivially $(i, j) \in \mathcal{B}$. \textit{(ii)} If $(i-1, j) \in \mathcal{A}$ and $\sum_{k=1}^{i-1} s_k \leq \sum_{k=1}^{j} d_k$ and $i-1 < p$, we distinguish multiple cases. If $i = p$ or $j = 1$, clearly $(i, j) \in \mathcal{B}$. So, let $j > 1$, $i < p$ and suppose for contradiction $(i, j) \notin \mathcal{B}$, that is, $\sum_{k=1}^{i} s_k \leq \sum_{k=1}^{j-1} d_k$. But this is not possible due to our structural property and the fact that $(i+1, j-1) \in [p] \times [q]$ lies strictly below the staircase. Consequently, $\sum_{k=1}^{i} s_k > \sum_{k=1}^{j-1} d_k$ and $(i, j) \in \mathcal{B}$. \textit{(iii)} Now, let $(i, j-1) \in \mathcal{A}$ and $\bigl(\sum_{k=1}^{i} s_k > \sum_{k=1}^{j-1} d_k$ or $i = p \bigr)$. Suppose for contradiction $(i, j) \notin \mathcal{B}$, that is, $i > 1$ and $\sum_{k=1}^{i-1} s_k > \sum_{k=1}^{j} d_k$. By the fact that the considered TP is balanced, this strict inequality can only hold if $j < q$. As before, we therefore obtain a contradiction to the structural property since $(i-1, j+1) \in [p] \times [q]$ is strictly above the staircase. 

        To prove $\mathcal{B} \subseteq \mathcal{A}$, let $(i, j) \in \mathcal{B}$ and suppose $(i, j) \notin \mathcal{A}$. Then $(i, j)$ must either be strictly below or above the staircase path given by $\mathcal{A}$. If it is strictly below, $i = 1$ cannot hold, so we must have $\sum_{k=1}^{i-1} s_k \leq \sum_{k=1}^{j} d_k$ by the definition of $\mathcal{B}$. However, this is an immediate contradiction to the structural property. If $(i, j)$ is strictly above the staircase, it is neither possible that $i = p$ nor that $j = 1$. Hence $\sum_{k=1}^{i} s_k > \sum_{k=1}^{j-1} d_k$, again contradicting the known structural insight.

        Altogether, we see that $\mathcal{A} = \mathcal{B}$.
    \end{proof}
\end{proposition}

\subsection{Compact formulas for dual solutions}\label{sec:TP_formulas}
Introducing dual variables $u_i$ for constraints~\eqref{transportation:c1} and $v_j$ for~\eqref{transportation:c2}, one obtains the following dual LP for problem~\eqref{transportation}.
\begin{maxi!}[2]
    {}
    {\sum_{i = 1}^p s_i u_i + \sum_{j = 1}^q d_j v_j}
    {\label{transportation_dual}}{}
    \addConstraint{u_i + v_j}{\leq c_{ij} \phantom{,}\quad}{\forall \, i \in [p], \, j \in [q] \label{transportation_dual:c1}}
    \addConstraint{u_i, v_j}{\in \mathbb{R} \phantom{,}\quad}{\forall \, i \in [p], \, j \in [q] \label{transportation_dual:c2}}
\end{maxi!}
Feasibility and boundedness of the primal LP~\eqref{transportation} imply, by strong duality, that the dual LP~\eqref{transportation_dual} attains a finite optimum. \citet{queyranne_general_1998} also provide an algorithm for solving the dual based on the output of the northwest-corner rule for the primal. This algorithm backtracks the cells $a^1, \dots, a^T$, thereby assigning values to the dual variables depending on whether the primal algorithm moved down or right when producing the northwest-corner solution. The precise steps are shown in Algorithm~\ref{dual_algo_backward} which assumes the output of Algorithm~\ref{NW_corner_rule} as input.

\begin{algorithm}[ht]
\caption{Dual backward recursion}\label{dual_algo_backward}
\DontPrintSemicolon
Declare $u \in \mathbb{R}^p$ and $v \in \mathbb{R}^q$\;
$v_q \gets 0$\;
$\eta_T \gets 1$ \tcp*[r]{$u_p$ will be set in the first iteration}
\For{$t$ from $T$ down to $1$}{
    \If{$\eta_t = 1$}{
        $u_{a_1^t} \gets c_{a^t} - v_{a_2^{t}}$\;
    }
    \ElseIf{$\eta_t = 2$}{
        $v_{a_2^t} \gets c_{a^t} - u_{a_1^{t}}$\;
    }
}
\Return{u, v}
\end{algorithm}

While Algorithm~\ref{dual_algo_backward} computes dual variables by backtracking the northwest-corner path from $(p,q)$ to $(1,1)$, we show in Lemma~\ref{lemma_correctness} below that the procedure can simply be reversed without compromising its correctness.

\begin{algorithm}[ht]
\caption{Dual forward recursion}\label{dual_algo_forward}
\DontPrintSemicolon
Declare $u \in \mathbb{R}^p$ and $v \in \mathbb{R}^q$\;
$u_1 \gets 0$\;
$\eta_0 \gets 2$ \tcp*[r]{$v_1$ will be set in the first iteration}
\For{$t$ from $1$ to $T$}{
    \If{$\eta_{t-1} = 1$}{
        $u_{a_1^t} \gets c_{a^t} - v_{a_2^{t}}$\;
    }
    \ElseIf{$\eta_{t-1} = 2$}{
        $v_{a_2^t} \gets c_{a^t} - u_{a_1^{t}}$\;
    }
}
\Return{u, v}
\end{algorithm}

\begin{lemma}[Correctness of Algorithm~\ref{dual_algo_forward}]\label{lemma_correctness}
    Let $(u, v)$ be the solution returned by Algorithm~\ref{dual_algo_backward}. Then, Algorithm~\ref{dual_algo_forward} returns $(u-u_1, v+u_1)$.\footnote{For a vector $y \in \mathbb{R}^d$ and a scalar $\alpha \in \mathbb{R}$, the sum $y + \alpha := (y_i + \alpha)_{i \in [d]}$ is understood componentwise.} This solution is feasible for \eqref{transportation_dual} and has the same objective value as $(u, v)$.
    \begin{proof}
        We show the first part of the lemma by induction. Let $(\tilde{u}, \tilde{v})$ be the solution returned by Algorithm~\ref{dual_algo_forward} after $t$ iterations. For $t = 1$, we have $(\tilde{u}_1, \tilde{v}_1) = (0, c_{11})$ and thus, immediately $\tilde{u}_1 = 0 = u_1 - u_1$. Moreover, $u_1 + v_1 = c_{11} = \tilde{v}_1$, so the claim holds after iteration $t = 1$. For the inductive step, suppose that the claim holds true in some iteration $t < T$. That is, $(\tilde{u}_{a_1^t}, \tilde{v}_{a_2^t}) = (u_{a_1^t} - u_1, v_{a_2^t} + u_1)$. For iteration $t+1$, there are two cases. If $\eta_t = 1$, the algorithm sets $\tilde{u}_{a_1^{t+1}} \gets c_{a^{t+1}} - \tilde{v}_{a_2^{t+1}}$. Since this corresponds to a \enquote{down} movement, we have $a^{t+1} = (a^t_1 + 1, a^t_2)$ implying $\tilde{v}_{a_2^{t+1}} = \tilde{v}_{a_2^{t}} = v_{a_2^t} + u_1$. Hence, $\tilde{u}_{a_1^{t+1}} = c_{a^{t+1}} - v_{a_2^t} - u_1 = u_{a_1^{t+1}} - u_1$. So, we have
        $(\tilde{u}_{a_1^{t+1}}, \tilde{v}_{a_2^{t+1}}) = (u_{a_1^{t+1}} - u_1, v_{a_2^{t+1}} + u_1)$. Analogously, one can show the case $\eta_t = 2$. Now, it follows by induction that Algorithm~\ref{dual_algo_forward} indeed returns $(u-u_1, v+u_1)$.

        Feasibility of $(u-u_1, v+u_1)$ is implied by the feasibility of $(u, v)$: for every $i \in [p]$ and $j \in [q]$, it holds $(u_i - u_1) + (v_j + u_1) = u_i + v_j \leq c_{ij}$. Moreover, the objective values coincide:
        \begin{equation*}
            \sum_{i = 1}^p s_i (u_i-u_1) + \sum_{j = 1}^q d_j (v_j + u_1) = \sum_{i = 1}^p s_i u_i + \sum_{j = 1}^q d_j v_j - u_1 \Bigg( \underbrace{\sum_{i = 1}^p s_i - \sum_{j = 1}^q d_j}_{=0} \Bigg)= \sum_{i = 1}^p s_i u_i + \sum_{j = 1}^q d_j v_j.
        \end{equation*}
        To conclude, the correctness of Algorithm~\ref{dual_algo_forward} follows from the correctness of Algorithm~\ref{dual_algo_backward}.
    \end{proof}
\end{lemma}

We remark that the initialization $u_1 \gets 0$ of Algorithm~\ref{dual_algo_forward} is somewhat arbitrary and could be replaced by any constant and one may even initialize $v_1$ instead of $u_1$ as the following lemma shows.

\begin{lemma}\label{lemma_initialization}
    Let $(\bar{u}, \bar{v})$ be the solution returned by Algorithm~\ref{dual_algo_forward} and let $b \in \mathbb{R}$. Then, the following holds.
    \begin{enumerate}
        \item[(i)] If the initialization is changed to $u_1 \gets b$, the output is $(\bar{u}+b, \bar{v}-b)$.
        \item[(ii)] If the initialization is changed to $v_1 \gets b, \, \eta_0 \gets 1$, the output is $(\bar{u} + (\bar{v}_1 - b), \bar{v} - (\bar{v}_1 - b))$.
    \end{enumerate}
    In both cases, the output remains feasible for \eqref{transportation_dual} with the same objective value as $(\bar{u}, \bar{v})$.
\end{lemma}

The proof of Lemma~\ref{lemma_initialization} follows the same argumentation used to prove Lemma~\ref{lemma_correctness}. It is therefore omitted.

For our application in Section~\ref{sec:application_to_domp}, we are particularly interested in optimal dual variables $(u, v)$. It will be helpful to rewrite the recursive formulas for $u$ and $v$ given in Algorithm~\ref{dual_algo_forward} into a compact expression that encodes the structure of the primal solution.

For $i \in \{2, \dots, p\}$, define $j_i := \min \{j \in [q] \, | \, (i,j) \in \{a^1, \dots, a^T\}\}$. As the northwest-corner rule moves in a staircase path from cell $(1,1)$ to cell $(p,q)$ while advancing either one row or one column per iteration, the existence of $j_i$ is guaranteed for all $i \in \{2, \dots, p\}$. Observe that Algorithm~\ref{dual_algo_forward} sets $u_i$ in iteration $t$ where $a^t = (i, j_i)$. Using this notation, the next theorem provides compact formulas for $u$ and $v$ computed by Algorithm~\ref{dual_algo_forward}. This is the first variant of our main result in Section~\ref{sec:transportation}.

\begin{theorem}\label{theorem_u_v_formulas}
    Algorithm~\ref{dual_algo_forward} returns the following solution: \footnote{We assume the convention that an empty sum has value equal to $0$.}
    \begin{align}
        &u_i = \sum_{k=2}^{i} \bigl( c_{k j_k} - c_{k-1, j_k} \bigr) &i \in [p], \label{u_formula} \\
        &v_j = c_{1j} + \sum_{\substack{k \in [p], \, k \geq 2 \\ j_k \leq j}} \bigl( c_{kj} - c_{k-1, j} - c_{k j_k} + c_{k-1, j_k} \bigr) &j \in [q]. \label{v_formula}
    \end{align}
    \begin{proof}
        First, we show \eqref{u_formula} by induction. For $i = 1$, the sum is empty, so $u_1 = 0$ which is coherent with how Algorithm~\ref{dual_algo_forward} sets $u_1$. For the inductive step suppose the formula is correct for some $i \in [p-1]$, that is, $u_i = \sum_{k=2}^{i} \bigl( c_{k j_k} - c_{k-1, j_k} \bigr)$. Recall that, at iteration $t$ with $a^t = (i, j_i)$, Algorithm~\ref{dual_algo_forward} sets $u_{i} \gets c_{i j_i} - v_{j_i}$. It proceeds by setting $v_{j'}$ for all $j' \in \{j_i+1, \dots, j_{i+1}\}$. (If this set is empty, then $j_{i+1} = j_i$ and the procedure continues directly with $u_{i+1}$.) Then, it assigns $u_{i+1} \gets c_{i+1, j_{i+1}} - v_{j_{i+1}}$. But $v_{j_{i+1}} = c_{i j_{i+1}} - u_i$, so $u_{i+1} = u_i + c_{i+1, j_{i+1}} - c_{i j_{i+1}} = \sum_{k = 2}^{i+1} \bigl( c_{k j_k} - c_{k-1, j_k} \bigr)$. We used the induction hypothesis in the last equation. This completes the inductive proof.

        To show \eqref{v_formula}, we make use of the previously shown formula \eqref{u_formula}. In particular, let $j \in [q]$ and $(i_j,j) \in \{a^1, \dots, a^T\}$. Then, Algorithm~\ref{dual_algo_forward} assigns $v_j \gets c_{i_j j} - u_{i_j}$. So, using \eqref{u_formula}, $v_j = c_{i_j j} - \sum_{k=2}^{i_j} \bigl( c_{k j_k} - c_{k-1, j_k} \bigr)$. For $i_j = 1$, it holds that $u_{i_j} = 0$, so the algorithm sets $v_j \gets c_{1j}$ which corresponds to the value given by the formula. For $i_j \geq 2$, observe that $i_j = \max \{k \in [p] \, | \, k \geq 2: j_k \leq j\}$. Therefore, we may rewrite $v_j = c_{i_j j} - \sum_{\substack{k \in [p], \, k \geq 2 \\ j_k \leq j}} \bigl( c_{k j_k} - c_{k-1, j_k} \bigr)$. Moreover, it is easy to see by a telescoping sum argument that $c_{i_j j} = c_{1j} + \sum_{\substack{k \in [p], \, k \geq 2 \\ j_k \leq j}} \bigl( c_{k j} - c_{k-1, j} \bigr)$. Putting everything together, we obtain $v_j = c_{1j} + \sum_{\substack{k \in [p], \, k \geq 2 \\ j_k \leq j}} \bigl( c_{kj} - c_{k-1, j} - c_{k j_k} + c_{k-1, j_k} \bigr)$.
    \end{proof}
\end{theorem}

With Theorem~\ref{theorem_u_v_formulas}, we have a compact formula for computing a solution to \eqref{transportation_dual}. The formula encodes the staircase path traversed by the corresponding primal solution obtained from the northwest-corner rule. This encoding is mainly based on the indices $j_i$ for all $i \in \{2, \dots, p\}$. Alternatively, one may consider an encoding based on indices $i_j$: for all $j \in \{1, \dots, q-1\}$, define $i_j := \max \{i \in [p] \, | \, (i,j) \in \{a^1, \dots, a^T\}\}$. If Algorithm~\ref{dual_algo_forward} is initialized with $v_1 \gets 0$ and $\eta_0 \gets 1$, it sets $v_j$ in iteration $t$ where $a^t = (i_j, j)$ and we obtain another pair of compact formulas for $u$ and $v$. These formulas constitute the second variant of our main result in this section.

\begin{theorem}\label{theorem_u_v_formulas_alternative}
    If Algorithm~\ref{dual_algo_forward} is initialized with $v_1 \gets 0$ and $\eta_0 \gets 1$, it returns the following solution:
    \begin{align}
        &u_i = c_{i1} + \sum_{\substack{k \in [q], \, k \geq 2 \\ i_{k-1} \leq i}} \bigl( c_{ik} - c_{i, k-1} - c_{i_{k-1} k} + c_{i_{k-1}, k-1} \bigr) &i \in [p], \label{u_formula_alternative} \\
        &v_j = \sum_{k=2}^{j} \bigl( c_{i_{k-1} k} - c_{i_{k-1}, k-1} \bigr) &j \in [q]. \label{v_formula_alternative}
    \end{align}
\end{theorem}

Theorem~\ref{theorem_u_v_formulas_alternative} can be proved in the same fashion as Theorem~\ref{theorem_u_v_formulas}. If $(\bar{u}, \bar{v})$ is the solution given by Theorem~\ref{theorem_u_v_formulas}, then, according to Lemma~\ref{lemma_initialization}, the values $(u,v)$ described by Theorem~\ref{theorem_u_v_formulas_alternative} satisfy $(u, v) = (\bar{u} + c_{11}, \bar{v} - c_{11})$. Moreover, if $(\bar{u}, \bar{v})$ is an optimal solution to \eqref{transportation_dual}, then so is $(u, v)$.

We continue with a useful characterization of the indices $j_i$, $i \in \{2, \dots, p\}$ and $i_j$, $j \in \{1, \dots, q-1\}$ that will be helpful for our application of the TP.
\begin{corollary}\label{corollary_get_indices}
    The following holds for $j_i$ and $i_j$ as obtained by Algorithm~\ref{NW_corner_rule}:
    \begin{enumerate}
        \item[(i)] For all $i \in \{2, \dots, p\}$, it holds $j_i = \min \{ j \in [q] \, | \, \sum_{k = 1}^{i-1} s_k \leq \sum_{k = 1}^{j} d_k \}$. 
        \item[(ii)] For all $j \in \{1, \dots, q-1\}$, it holds $i_j = \max \{ i \in [p] \, | \, \sum_{k = 1}^{i-1} s_k \leq \sum_{k = 1}^{j} d_k \}$.
    \end{enumerate}
    \begin{proof}
        To prove $(i)$, let $i \in \{2, \dots, p\}$. By definition of $j_i$ and by Proposition~\ref{NW_corner_characterization_2}, we obtain
        {\allowdisplaybreaks
        \begin{align*}
            j_i &= \min \{j \in [q] \, | \, (i,j) \in \{a^1, \dots, a^T\}\} \\
            &\overset{\eqref{NW_characterization_equation}}{=} \min \bigg\{j \in [q] \, \bigg| \, \bigg( \sum_{k=1}^{i-1} s_k \leq \sum_{k=1}^{j} d_k \, \lor \, i = 1 \bigg) \land \bigg( \sum_{k=1}^{i} s_k > \sum_{k=1}^{j-1} d_k \, \lor \, i = p \, \lor \, j = 1 \bigg) \bigg\} \\
            &= \min \bigg\{j \in [q] \, \bigg| \, \sum_{k=1}^{i-1} s_k \leq \sum_{k=1}^{j} d_k \bigg\}.
        \end{align*}
        }
        For the last equation, observe that the only condition that could possibly prevent the minimum from choosing $j = 1$ is the lower bound given by $\sum_{k=1}^{i-1} s_k \leq \sum_{k=1}^{j} d_k$. 

        To prove $(ii)$, let $j \in \{1, \dots, q-1\}$. By definition of $i_j$ and by Proposition~\ref{NW_corner_characterization_2}, we obtain
        {\allowdisplaybreaks
        \begin{align*}
            i_j &= \max \{i \in [p] \, | \, (i,j) \in \{a^1, \dots, a^T\}\} \\
            &\overset{\eqref{NW_characterization_equation}}{=} \max \bigg\{i \in [p] \, \bigg| \, \bigg( \sum_{k=1}^{i-1} s_k \leq \sum_{k=1}^{j} d_k \, \lor \, i = 1 \bigg) \land \bigg( \sum_{k=1}^{i} s_k > \sum_{k=1}^{j-1} d_k \, \lor \, i = p \, \lor \, j = 1 \bigg) \bigg\} \\
            &= \max \bigg\{i \in [p] \, \bigg| \, \sum_{k=1}^{i-1} s_k \leq \sum_{k=1}^{j} d_k \bigg\}.
        \end{align*}
        }
        Similar as before, the only condition preventing the maximum from choosing $i = p$ is $\sum_{k=1}^{i-1} s_k \leq \sum_{k=1}^{j} d_k$, which explains the last equation.
    \end{proof}
\end{corollary}

From the preceding discussion, it is already implied that the vectors $u, v$ defined by our formulas from Theorem~\ref{theorem_u_v_formulas} and~\ref{theorem_u_v_formulas_alternative} are dual feasible. Remarkably, feasibility is preserved even when the indices $j_i$ (respectively $i_j$) do not encode the staircase path traced by Algorithm~\ref{NW_corner_rule}. In fact, for any arbitrary choice of indices $(j_i)_{i \in [p]}$ in $[q]$ (and analogously $(i_j)_{j \in [q]}$ in $[p]$), $(u, v)$ remains a feasible solution for the dual~\eqref{transportation_dual}. We conclude this section by proving this result.

\begin{lemma}\label{u_v_dual_feasibility}
    Let $(j_i)_{i \in [p]}$ be any sequence of indices in $[q]$. Using these indices, define vectors $u \in \mathbb{R}^p$ and $v \in \mathbb{R}^q$ according to \eqref{u_formula} and \eqref{v_formula}, respectively. Then, $(u, v)$ is feasible for \eqref{transportation_dual}, that is, $u_i + v_j \leq c_{ij}$ for all $i \in [p]$, $j \in [q]$.
    
    \noindent An analogous result holds for the formulas \eqref{u_formula_alternative} and \eqref{v_formula_alternative} with indices $(i_j)_{j \in [q]}$ in $[p]$.
    
    \begin{proof}
        Fix $i \in [p]$ and $j \in [q]$. With the identity $c_{ij} = c_{1j} + \sum_{k=2}^{i} \bigl( c_{k j} - c_{k-1, j} \bigr)$, we obtain
        {\allowdisplaybreaks
        \begin{align*}
            c_{ij} - (u_i + v_j) &= c_{1j} + \sum_{k=2}^{i} \bigl( c_{k j} - c_{k-1, j} \bigr) - \Biggl( \sum_{k=2}^{i} \bigl( c_{k j_k} - c_{k-1, j_k} \bigr) +\\
            &\phantom{=} c_{1j} + \sum_{\substack{k \in [p], \, k \geq 2 \\ j_k \leq j}} \bigl( c_{kj} - c_{k-1, j} - c_{k j_k} + c_{k-1, j_k} \bigr) \Biggr) \\
            &= \sum_{k=2}^{i} \bigl( c_{k j} - c_{k-1, j} - c_{k j_k} + c_{k-1, j_k} \bigr) - \sum_{\substack{k \in [p], \, k \geq 2 \\ j_k \leq j}} \bigl( c_{kj} - c_{k-1, j} - c_{k j_k} + c_{k-1, j_k} \bigr) \\
            &= \sum_{\substack{k \in [i], \, k \geq 2 \\ j_k > j}} \underbrace{\bigl( c_{k j} - c_{k-1, j} - c_{k j_k} + c_{k-1, j_k} \bigr)}_{\geq 0} - \sum_{\substack{k \in [p], \, k > i \\ j_k < j}} \underbrace{\bigl( c_{kj} - c_{k-1, j} - c_{k j_k} + c_{k-1, j_k} \bigr)}_{\leq 0} \\
            &\geq 0.
        \end{align*}
        }
        The cancellation in the third equality follows by partitioning the indices: terms with $2 \leq k \leq i$ and $j_k \leq j$ appear in both sums and subtract to zero. The final inequality holds because the Monge property implies that $c_{kj} - c_{k-1, j} - c_{k j_k} + c_{k-1, j_k}$ is non-negative when $j_k > j$ and non-positive when $j_k < j$.
    \end{proof}
\end{lemma}

\section{Application to the Discrete Ordered Median Problem}\label{sec:application_to_domp}
We now apply the formulas derived in Section~\ref{sec:transportation}. First, in Subsection~\ref{sec:DOMP_definition}, we introduce all notation needed to present the DOMP and reproduce a specific DOMP formulation by \citet{labbe_comparative_2017}. In Subsection~\ref{sec:Benders}, we show how the assumption of non-increasing objective function weights allows us to simplify the model from the literature. By applying Benders decomposition to the resulting formulation, we derive two new reformulations of the DOMP for non-increasing objective weight vectors. Here, the formulas from the previous section play the key role in obtaining explicit Benders cut expressions. We conclude the section with theoretical and practical remarks (Subsections~\ref{sec:theoretical_notes} and \ref{sec:practical_notes}).

\subsection{Problem definition and formulation from the literature}\label{sec:DOMP_definition}
Let $n \in \mathbb{N}$ be the number of client locations and assume without loss of generality that new facilities can be placed exactly at these locations. Furthermore, let $p \in \mathbb{N}$, $p \leq n$ be the number of facilities that is to be opened. Next, let $c \in \mathbb{R}_{\geq 0}^{n \times n}$ be the cost matrix, where entry $c_{ij}$ indicates the cost of satisfying all the demand of client $i$ from facility $j$. Every customer should be assigned to and fully served by the closest open facility. This implicitly assumes that there are no capacity limits for the opened facilities.

Before explaining the DOMP objective, we  recall the definition of the general ordered median function.
\begin{definition}[Ordered median function {\citep{nickel_puerto_2005}}]
    For a given $x \in \mathbb{R}^n$, let $x_\leq := (x_{(1)}, x_{(2)}, ..., x_{(n)})$ be a re-ordering of $x$ such that $x_{(1)} \leq x_{(2)} \leq \cdots \leq x_{(n)}$.
    A function $f: \mathbb{R}^n \rightarrow \mathbb{R}$ is defined to be an ordered median function, if there exists a vector $\lambda \in \mathbb{R}^n$ such that 
    \[f(x) = \langle \lambda, x_\leq \rangle = \sum_{i = 1}^n \lambda_i x_{(i)} \quad \forall \, x \in \mathbb{R}^n.\]
\end{definition}
The DOMP objective function is an ordered median function applied to the vector of customer allocation costs. More precisely, let $Y \subseteq [n]$ be the set of opened facilities and denote the allocation cost for each customer $i$ by $c_i(Y) = \min_{j \in Y} c_{ij}$. Then, $c(Y) := \left(c_i(Y)\right)_{i \in [n]}$ is the allocation cost vector that is to be minimized in the ordered median objective. That is, the DOMP with weight vector $\lambda \in \mathbb{R}^n$ and associated ordered median function $f_\lambda$ can be stated as
\begin{equation}\label{DOMP_definition}
    \min_{\substack{Y \subseteq [n] \\ |Y| = p}} f_\lambda\big(c(Y)\big).
\end{equation}

Note that we allow a slight overlap between the notation used for the TP and the DOMP in order to adhere to standard notation (e.g., the cost matrix $c$); the meaning will always be clear from the context.

To facilitate the presentation of our models, we will make use of the following notation later in this section. For a vector $x \in \mathbb{R}^d$, define $\Delta_k^{x} := x_k - x_{k-1}$ for $k \in [d]$ and $x_0 := 0$.

In the literature, one can find plenty of different ways to model \eqref{DOMP_definition} as a MILP \citep[e.g.,][]{labbe_comparative_2017, marin_fresh_2020, ljubic_benders_2024}. Many of the best-performing models are based on two-index variables $x_{ij} \in \{0, 1\}$ taking value $1$ if and only if client $i$ is served by facility $j$. Furthermore, these models typically use indicator variables $y_j \in \{0,1\}$ for each facility $j$ that are equal to $1$ if and only if facility $j$ is opened. Any feasible solution to the considered ordered median location problem must lie inside the $p$-median polytope which is defined by
\[
X := \left\{\begin{aligned}
  &x \in [0,1]^{n \times n}, \; y \in [0,1]^n \, \bigg| \, \; 
  \sum_{j=1}^n y_{j} = p; \; \sum_{j=1}^n x_{ij} = 1 \; \forall \, i \in [n]; \;
  x_{ij} \leq y_j \; \forall \, i,j \in [n]; \\
  &\sum_{\substack{j \in [n] \\ c_{ij} > c_{im}}} x_{ij} + y_m \leq 1 
  \; \forall \, i, m \in [n]
\end{aligned}\right\}.
\]
The inequalities defining $X$ ensure that exactly $p$ facilities are opened and that each client is assigned to exactly one open facility, such that there exists no closer open facility. In the following, we will write $X_I := X \cap (\mathbb{Z}^{n\times n} \times \mathbb{Z}^n)$ for the intersection of the $p$-median polytope with the integer lattice points.

This paper focuses on the \enquote{staircase model} proposed in \cite{labbe_comparative_2017}, which we will refer to as $\text{DOMP}_{s}$ below. In order to present their formulation, \citet{labbe_comparative_2017} sort the $g$ distinct positive entries of the allocation cost matrix $c$ increasingly and insert an additional zero as the first element. In particular, they write $0 =: c_{(0)} < c_{(1)} < \cdots < c_{(g)} = \max_{i,j \in [n]} c_{ij}$. The model makes use of binary variables $w_{lh}$ that indicate whether the allocation cost in position $l$ is equal to $c_{(h)}$ or not for each $l \in [n]$ and $h \in [g] \cup \{0\}$. For convenience, define $[g]_0 := [g] \cup \{0\}$.

\begin{mini!}[2]<b>
    {}
    {\sum_{h = 0}^g \sum_{l = 1}^n \lambda_l \, c_{(h)} \, w_{lh} \label{DOMP_s:obj}}
    {\label{DOMP_s}}{\text{DOMP}_{s}:\,}
    \addConstraint{\sum_{h = 0}^g w_{lh}}{= 1 \phantom{,}\quad}{\forall \, l \in [n] \label{DOMP_s:c1}}
    \addConstraint{\sum_{l = 1}^n w_{lh}}{= \sum_{\substack{i,j \in [n] \\ c_{ij} = c_{(h)}}} x_{ij} \phantom{,}\quad}{\forall \, h \in [g]_0 \label{DOMP_s:c2}}
    \addConstraint{\sum_{h' < h} w_{l+1,h'} + \sum_{h' \geq h} w_{lh'}}{\leq 1 \phantom{,}\quad}{\forall \, l \in [n-1], \, h \in [g] \label{DOMP_s:c3}}
    \addConstraint{w_{lh}}{\in \{0,1\} \phantom{,}\quad}{\forall \, l \in [n],\, h \in [g]_0 \label{DOMP_s:c4}}
    \addConstraint{(x, y)}{\in X_I \label{DOMP_s:c5}}
\end{mini!}
Constraints~\eqref{DOMP_s:c1} ensure that each position $l$ in the ordered allocation cost sequence will be occupied by exactly one value $c_{(h)}$. By constraints~\eqref{DOMP_s:c2}, the number of occurrences of each $c_{(h)}$ in the ordered cost sequence will exactly correspond to the number of clients whose allocation cost equals $c_{(h)}$. Next, constraints~\eqref{DOMP_s:c3} ensure that the allocation cost sequence given by the variables $w_{lh}$ is sorted non-decreasingly. Due to the correct definition of the variables $w_{lh}$ by means of constraints~\eqref{DOMP_s:c1} -- \eqref{DOMP_s:c4}, the objective function~\eqref{DOMP_s:obj} correctly computes the ordered $\lambda$-weighted sum of allocation costs which is to be minimized.

\subsection{Benders reformulations for non-increasing objective vectors}\label{sec:Benders}
For our application of the TP, we study non-increasing objective vectors $\lambda$. We start by observing the possibility to simplify model~\eqref{DOMP_s} in that case.

\begin{lemma}\label{lemma_monotone_staircase_model}
    For any $\lambda \in \mathbb{R}^n$ satisfying $\lambda_1 \geq \lambda_2 \geq \cdots \geq \lambda_n$, the following compact model is a valid DOMP formulation.
    \begin{mini!}[2]<b>
        {}
        {\sum_{h = 0}^g \sum_{l = 1}^n \lambda_l \, c_{(h)} \, w_{lh} \label{mDOMP_s:obj}}
        {\label{mDOMP_s}}{\text{mDOMP}_{s}:\,}
        \addConstraint{\sum_{h = 0}^g w_{lh}}{= 1 \phantom{,}\quad}{\forall \, l \in [n] \label{mDOMP_s:c1}}
        \addConstraint{\sum_{l = 1}^n w_{lh}}{= \sum_{\substack{i,j \in [n] \\ c_{ij} = c_{(h)}}} x_{ij} \phantom{,}\quad}{\forall \, h \in [g]_0 \label{mDOMP_s:c2}}
        \addConstraint{w_{lh}}{\geq 0 \phantom{,}\quad}{\forall \, l \in [n],\, h \in [g]_0 \label{mDOMP_s:c3}}
        \addConstraint{(x, y)}{\in X_I \label{mDOMP_s:c4}}
    \end{mini!}
    \begin{proof}
        First, we show that the special structure of $\lambda$ renders the ordering constraints~\eqref{DOMP_s:c3} redundant. It is well-known that the scalar product of a non-increasing vector $\lambda$ with a vector $x$ rearranged according to some permutation $\pi$ is minimal if $\pi$ sorts $x$ non-decreasingly \citep[][Inequality~368]{inequalities}. More precisely, if $x_\pi$ denotes the reordering of $x \in \mathbb{R}^n$ according to the permutation $\pi$ and $x_\leq$ denotes the non-decreasingly sorted $x$, then it holds $\langle \lambda, x_\leq \rangle \leq \langle \lambda, x_\pi \rangle$ for every permutation $\pi: [n] \rightarrow [n]$. Now, let $(w^*, x^*, y^*)$ be an optimal solution to \eqref{DOMP_s} excluding constraints~\eqref{DOMP_s:c3} and consider the vector of optimal allocation costs $(c_{(h_l)})_{l \in [n]}$, where $h_l \in [g]_0$ for $l \in [n]$ is such that $w^*_{lh_l} = 1$. The existence of such $h_l$ for each $l \in [n]$ is guaranteed by constraints~\eqref{DOMP_s:c1} and~\eqref{DOMP_s:c4}. According to the above stated rearrangement inequality, the minimization automatically ensures that $c_{(h_l)} \leq c_{(h_{l+1})}$ and thus, $h_l \leq h_{l+1}$ for each $l \in [n-1]$. So, for every $l \in [n-1]$ and $h \leq h_l$, we have $\sum_{h' < h} w^*_{l+1,h'} = 0$ and for $h > h_l$, it holds that $\sum_{h' \geq h} w^*_{l,h'} = 0$. Hence, constraints~\eqref{DOMP_s:c3} are all satisfied.

        Next, the upper bounds $w_{lh} \leq 1$ for every $l \in [n]$ and $h \in [g]_0$ are implied by \eqref{DOMP_s:c1}, so \eqref{DOMP_s:c4} can be replaced by $w_{lh} \geq 0, \, w_{lh} \in \mathbb{Z}$ for all $l \in [n], \, h \in [g]_0$.

        Finally, the integrality restrictions on the $w$-variables can be relaxed. Consider the subproblem \eqref{mDOMP_s:obj} -- \eqref{mDOMP_s:c3} for any fixed $(\bar{x}, \bar{y}) \in X_I$. Since $\sum_{h=0}^g \sum_{\substack{i,j \in [n] \\ c_{ij} = c_{(h)}}} \bar{x}_{ij} = \sum_{i, j \in [n]} \bar{x}_{ij} = n$, this is a balanced TP with integer supplies and demands. As mentioned in Section~\ref{sec:transportation}, such problems are feasible, bounded, and every basic feasible solution is integral.
        
        Overall, we see that model~\eqref{mDOMP_s} always admits an optimal solution $(w^*, x^*, y^*)$ which is feasible for model~\eqref{DOMP_s}. Since the feasible region of \eqref{DOMP_s} is a subset of the feasible region of \eqref{mDOMP_s}, the solution $(w^*, x^*, y^*)$ is optimal for \eqref{DOMP_s}.
    \end{proof}
\end{lemma}

Following an approach presented by \citet{ljubic_benders_2024}, we can derive a Benders reformulation of the staircase model~\eqref{mDOMP_s} for non-increasing $\lambda$ by projecting out the variables $w$ and replacing the objective function by a new variable $\theta$. Let $(\bar{x}, \bar{y}) \in X_I$ and define $\bar{x}_h := \sum_{\substack{i,j \in [n] \\ c_{ij} = c_{(h)}}} \bar{x}_{ij}$ for each $h \in [g]_0$. Then, we are facing the following subproblem.

\begin{mini!}[2]<b>
    {}
    {\sum_{h = 0}^g \sum_{l = 1}^n \lambda_l \, c_{(h)} \, w_{lh}}
    {\label{staircase_sub}}{}
    \addConstraint{\sum_{h = 0}^g w_{lh}}{= 1 \phantom{,}\quad}{\forall \, l \in [n] \label{staircase_sub:c1}}
    \addConstraint{\sum_{l = 1}^n w_{lh}}{= \bar{x}_h \phantom{,}\quad}{\forall \, h \in [g]_0 \label{staircase_sub:c2}}
    \addConstraint{w_{lh}}{\geq 0 \phantom{,}\quad}{\forall \, l \in [n],\, h \in [g]_0 \label{staircase_sub:c3}}
\end{mini!}

As already observed in the proof of Lemma~\ref{lemma_monotone_staircase_model}, this is a balanced TP with integral right-hand sides, which is feasible and bounded. By introducing dual variables $u_l$ for constraints~\eqref{staircase_sub:c1} and $v_h$ for~\eqref{staircase_sub:c2}, the dual LP is as follows.

\begin{maxi!}[2]
    {}
    {\sum_{l = 1}^n u_l + \sum_{h = 0}^g \bar{x}_h v_h}
    {\label{staircase_dual}}{}
    \addConstraint{u_l + v_h}{\leq \lambda_{l} c_{(h)} \phantom{,}\quad}{\forall \, l \in [n], h \in [g]_0 \label{staircase_dual:c1}}
    \addConstraint{u_l, v_h}{\in \mathbb{R} \phantom{,}\quad}{\forall \, l \in [n], h \in [g]_0 \label{staircase_dual:c2}}
\end{maxi!}
For any dual feasible solution $(u,v)$, we obtain a Benders cut of the form
\begin{equation}
    \theta \geq \sum_{l = 1}^n u_l + \sum_{h = 0}^g v_h \sum_{\substack{i,j \in [n] \\ c_{ij} = c_{(h)}}} x_{ij}. \label{benders_optimality_cuts}
\end{equation}

Observe that the matrix $\left(\lambda_{l} c_{(h)}\right)_{l \in [n], \, h \in [g]_0}$ satisfies the Monge property: for $1 \leq l < l' \leq n$ and $0 \leq h < h' \leq g$, we see that $\lambda_{l} c_{(h')} + \lambda_{l'} c_{(h)} - (\lambda_{l} c_{(h)} + \lambda_{l'} c_{(h')}) = (\lambda_l - \lambda_{l'})(c_{(h')} - c_{(h)}) \geq 0$. Thus, $\lambda_{l} c_{(h)} + \lambda_{l'} c_{(h')} \leq \lambda_{l} c_{(h')} + \lambda_{l'} c_{(h)}$ holds true.

Therefore, the theory from Section~\ref{sec:transportation} applies. Particularly, the formulas given in Theorem~\ref{theorem_u_v_formulas} and Theorem~\ref{theorem_u_v_formulas_alternative} yield optimal dual variables $u$ and $v$ that can be inserted into the Benders optimality cut expression \eqref{benders_optimality_cuts}. Using Theorem~\ref{theorem_u_v_formulas}, we obtain the following exponential-sized model for DOMP with non-increasing objective vector.

\begin{theorem}\label{transport_domp_formulation}
    For any $\lambda \in \mathbb{R}^n$ satisfying $\lambda_1 \geq \lambda_2 \geq \cdots \geq \lambda_n$, the following exponential-sized model is a valid DOMP formulation.
    \begin{mini!}[2]<b>
        {}
        {\theta}
        {\label{mDOMP_sB1}}{\text{mDOMP}_{sB1}:\,}
        \addConstraint{\theta}{\geq \sum_{k = 1}^n \Delta_k^\lambda \bigg( (n-k+1) c_{(f(k))} +  \phantom{,}\quad}{\nonumber}
        \addConstraint{\phantom{,}\quad}{ \sum_{\substack{i,j \in [n] \\ c_{ij} > c_{(f(k))}}} (c_{ij} -  c_{(f(k))}) x_{ij} \bigg) \phantom{,}\quad}{ 
        \begin{array}{l}
            \forall \, f: [n] \rightarrow [g]_0 \text{ mon. incr.} \\
            \text{with } f(1) = 0
        \end{array} \label{mDOMP_sB1:c1}}
        \addConstraint{\theta}{\in \mathbb{R} \phantom{,}\quad}{\label{mDOMP_sB1:c2}}
        \addConstraint{(x, y)}{\in X_I \phantom{,}\quad}{\label{mDOMP_sB1:c3}}
    \end{mini!}
    
    \begin{proof}
        For a monotonically increasing mapping $f: [n] \rightarrow [g]_0$ with $f(1) = 0$, define the vectors $u^f, v^f$ as follows:
        \begin{align}\label{u_f_v_f_formulas}
            \begin{alignedat}{2}
                &u^f_l = \sum_{k=2}^{l} \bigl( \lambda_k - \lambda_{k-1} \bigr) c_{(f(k))}& \quad &l \in [n],\\
                &v^f_{h} = \sum_{\substack{k \in [n] \\ f(k) < h}} \bigl( \lambda_{k} - \lambda_{k-1} \bigr) \bigl( c_{(h)} - c_{(f(k))} \bigr) & \quad &h \in [g]_0.
            \end{alignedat}
        \end{align}

        Note that $v^f_{h} = \lambda_1 c_{(h)} + \sum_{\substack{k \in [n], \, k \geq 2 \\ f(k) \leq h}} \bigl( \lambda_{k} - \lambda_{k-1} \bigr) \bigl( c_{(h)} - c_{(f(k))} \bigr)$ since $\lambda_1 c_{(h)} = \bigl( \lambda_{1} - \lambda_{0} \bigr) \bigl( c_{(h)} - c_{(f(1))} \bigr)$ and $c_{(h)} - c_{(f(k))} = 0$ whenever $f(k) = h$. So, the above formulas \eqref{u_f_v_f_formulas} correspond to those given in Theorem~\ref{theorem_u_v_formulas} for the optimal dual variables $(u, v)$ and $f(k)$ corresponds to the indices $j_i$ encoding a potential staircase path that Algorithm~\ref{NW_corner_rule} may traverse. By Lemma~\ref{u_v_dual_feasibility}, $(u^f, v^f)$ is a feasible solution to \eqref{staircase_dual}. Furthermore, since every valid staircase path translates to a monotonically increasing function $f$, no cuts are omitted in \eqref{mDOMP_sB1:c1}.
        
        Therefore, by Theorem~\ref{theorem_u_v_formulas} and the Benders optimality cut formula \eqref{benders_optimality_cuts}, it remains to show that for any monotonically increasing $f: [n] \rightarrow [g]_0$ with $f(1) = 0$, it holds 
        \begin{equation}\label{u_v_sum}
            \sum_{l = 1}^n u^f_l + \sum_{h = 0}^g v^f_h \sum_{\substack{i,j \in [n] \\ c_{ij} = c_{(h)}}} x_{ij} = \sum_{k = 1}^n \Delta_k^\lambda \bigg( (n-k+1) c_{(f(k))} + \sum_{\substack{i,j \in [n] \\ c_{ij} > c_{(f(k))}}} (c_{ij} -  c_{(f(k))}) x_{ij} \bigg).
        \end{equation}
        Through rearrangement of terms, it is straightforward to verify that
        \begin{equation}\label{u_sum}
            \sum_{l = 1}^n u^f_l = \sum_{k = 1}^n \Delta_k^\lambda (n-k+1) c_{(f(k))}.
        \end{equation}
        Furthermore, we can rewrite
        {\allowdisplaybreaks
        \begin{alignat}{2}\label{v_sum}
            \sum_{h = 0}^g v^f_h \sum_{\substack{i,j \in [n] \\ c_{ij} = c_{(h)}}} x_{ij} &= \sum_{h = 0}^g \sum_{\substack{k \in [n] \\ f(k) < h}} \Delta_k^\lambda \bigl( c_{(h)} - c_{(f(k))} \bigr) \sum_{\substack{i,j \in [n] \\ c_{ij} = c_{(h)}}} x_{ij} \nonumber \\
            &= \sum_{k=1}^n \Delta_k^\lambda \sum_{\substack{h \in [g]_0 \\ h > f(k)}} \sum_{\substack{i,j \in [n] \\ c_{ij} = c_{(h)}}} \bigl( c_{(h)} - c_{(f(k))} \bigr) x_{ij} \nonumber \\
            &= \sum_{k=1}^n \Delta_k^\lambda \sum_{\substack{i,j \in [n] \\ c_{ij} > c_{(f(k))}}} \bigl( c_{ij} - c_{(f(k))} \bigr) x_{ij}.
        \end{alignat}
        }
        In the last equation, we used that $c_{(\cdot)}$ is strictly increasing: $h > f(k)$ and $c_{ij} = c_{(h)}$ imply that $c_{ij} > c_{(f(k))}$. Conversely, if $c_{ij} > c_{(f(k))}$, then there exists $h \in [g]_0$, $h > f(k)$ such that $c_{ij} = c_{(h)}$.
        
        Finally, adding \eqref{u_sum} and \eqref{v_sum} proves \eqref{u_v_sum}.
    \end{proof}
\end{theorem}

If we use the formulas from Theorem~\ref{theorem_u_v_formulas_alternative} instead of the ones from Theorem~\ref{theorem_u_v_formulas}, we obtain another exponential-sized model for DOMP with non-increasing objective vector.

\begin{theorem}\label{transport_domp_formulation_alternative}
    For any $\lambda \in \mathbb{R}^n$ satisfying $\lambda_1 \geq \lambda_2 \geq \cdots \geq \lambda_n$, the following exponential-sized model is a valid DOMP formulation.
    \begin{mini!}[2]<b>
        {}
        {\theta}
        {\label{mDOMP_sB2}}{\text{mDOMP}_{sB2}:\,}
        \addConstraint{\theta}{\geq \sum_{k = 0}^{g-1} \Delta_{k+1}^{c_{(\cdot)}} \bigg( \sum_{\substack{l \in [n] \\ l > f(k)}} (\lambda_l - \lambda_{f(k)})  \phantom{,}\quad}{\nonumber}
        \addConstraint{\phantom{,}\quad}{ + \lambda_{f(k)} \sum_{\substack{i,j \in [n] \\ c_{ij} > c_{(k)}}} x_{ij} \bigg) \phantom{,}\quad}{ 
        \begin{array}{l}
            \forall \, f: [g-1]_0 \rightarrow [n] \text{ mon. incr.}
        \end{array} \label{mDOMP_sB2:c1}}
        \addConstraint{\theta}{\in \mathbb{R} \phantom{,}\quad}{\label{mDOMP_sB2:c2}}
        \addConstraint{(x, y)}{\in X_I \phantom{,}\quad}{\label{mDOMP_sB2:c3}}
    \end{mini!}
    
    \begin{proof}
        For a monotonically increasing mapping $f: [g-1]_0 \rightarrow [n]$, define the vectors $u^f, v^f$ as follows:
        \begin{align}\label{u_f_v_f_formulas_alternative}
            \begin{alignedat}{2}
                &u^f_l = \sum_{\substack{k \in [g-1]_0 \\ f(k) < l}} \bigl( \lambda_l - \lambda_{f(k)} \bigr) (c_{(k+1)} - c_{(k)})& \quad &l \in [n],\\
                &v^f_{h} = \sum_{k = 0}^{h-1} \lambda_{f(k)} \bigl( c_{(k+1)} - c_{(k)} \bigr) & \quad &h \in [g]_0.
            \end{alignedat}
        \end{align}
        
        A simple index shift shows that $u^f_{l} = \lambda_l c_{(0)} + \sum_{\substack{k \in [g] \\ f(k-1) \leq l}} \bigl( \lambda_{l} - \lambda_{f(k-1)} \bigr) \bigl( c_{(k)} - c_{(k-1)} \bigr)$ and that $v^f_{h} = \sum_{k = 1}^{h} \lambda_{f(k-1)} \bigl( c_{(k)} - c_{(k-1)} \bigr)$. So, the above formulas \eqref{u_f_v_f_formulas_alternative} correspond exactly to those given in Theorem~\ref{theorem_u_v_formulas_alternative} for the optimal dual variables $(u, v)$ and $f(k)$ corresponds to the indices $i_j$ encoding a potential staircase path that Algorithm~\ref{NW_corner_rule} may traverse. By Lemma~\ref{u_v_dual_feasibility}, $(u^f, v^f)$ is a feasible solution to \eqref{staircase_dual}. With the same argument as in the proof of Theorem~\ref{transport_domp_formulation}, the set of cuts in \eqref{mDOMP_sB2:c1} is complete.
        
        Therefore, by Theorem~\ref{theorem_u_v_formulas_alternative} and the Benders optimality cut formula \eqref{benders_optimality_cuts}, it remains to show that for any monotonically increasing $f: [g]_0 \rightarrow [n]$ with $f(g) = n$, it holds 
        \begin{equation}\label{u_v_sum_alternative}
            \sum_{l = 1}^n u^f_l + \sum_{h = 0}^g v^f_h \sum_{\substack{i,j \in [n] \\ c_{ij} = c_{(h)}}} x_{ij} = \sum_{k = 0}^{g-1} \Delta_{k+1}^{c_{(\cdot)}} \bigg( \sum_{\substack{l \in [n] \\ l > f(k)}} (\lambda_l - \lambda_{f(k)}) + \lambda_{f(k)} \sum_{\substack{i,j \in [n] \\ c_{ij} > c_{(k)}}} x_{ij} \bigg).
        \end{equation}
        By definition of $u^f_l$, it holds that
        \begin{equation}\label{u_sum_alternative}
            \sum_{l = 1}^n u^f_l = \sum_{k = 0}^{g-1} \Delta_{k+1}^{c_{(\cdot)}} \sum_{\substack{l \in [n] \\ l > f(k)}} (\lambda_l - \lambda_{f(k)}).
        \end{equation}
        Furthermore, we can rewrite
        {\allowdisplaybreaks
        \begin{alignat}{2}\label{v_sum_alternative}
            \sum_{h = 0}^g v^f_h \sum_{\substack{i,j \in [n] \\ c_{ij} = c_{(h)}}} x_{ij} &= \sum_{h = 0}^g \sum_{k = 0}^{h-1} \lambda_{f(k)} \big(c_{(k+1)} - c_{(k)} \big) \sum_{\substack{i,j \in [n] \\ c_{ij} = c_{(h)}}} x_{ij} \nonumber \\
            & = \sum_{h = 1}^{g} \lambda_{f(h-1)} \big(c_{(h)} - c_{(h-1)} \big) \sum_{h' = h}^{g} \sum_{\substack{i,j \in [n] \\ c_{ij} = c_{(h')}}} x_{ij} \nonumber \\
            & = \sum_{h = 0}^{g-1} \lambda_{f(h)} \big(c_{(h+1)} - c_{(h)} \big) \sum_{h' = h+1}^{g} \sum_{\substack{i,j \in [n] \\ c_{ij} = c_{(h')}}} x_{ij} \nonumber \\
            & = \sum_{h = 0}^{g-1} \lambda_{f(h)} \big(c_{(h+1)} - c_{(h)} \big) \sum_{\substack{i,j \in [n] \\ c_{ij} > c_{(h)}}} x_{ij}.
        \end{alignat}
        }
        In the last equation, we used that $c_{(\cdot)}$ is strictly increasing. More precisely, $h' > h$ and $c_{ij} = c_{(h')}$ imply that $c_{ij} > c_{(h)}$. Conversely, if $c_{ij} > c_{(h)}$, then there exists $h' \in [g]_0$, $h' > h$ such that $c_{ij} = c_{(h')}$.
        
        Finally, adding \eqref{u_sum_alternative} and \eqref{v_sum_alternative} proves \eqref{u_v_sum_alternative}.
    \end{proof}
\end{theorem}

\subsection{Theoretical notes}\label{sec:theoretical_notes}
It is clear that both $\text{mDOMP}_{sB1}$ and $\text{mDOMP}_{sB2}$ involve exponentially many constraints. A closer inspection shows that the number of constraints is the same for both models.
\begin{lemma}\label{lemma_num_constraints}
    The number of constraints of type \eqref{mDOMP_sB1:c1} and \eqref{mDOMP_sB2:c1} is $\binom{n + g - 1}{n - 1}$.
    \begin{proof}
        Using standard counting techniques, one can show that the number of monotonically increasing functions $f: [a] \rightarrow [b]$ is equal to $\binom{a + b - 1}{a}$ \citep[see, e.g.,][]{stanley1997enumerative}. For \eqref{mDOMP_sB1:c1}, we have $a = n-1$ and $b = g+1$, so the number of constraints is $\binom{n + g - 1}{n - 1}$. Analogously, with $a = g$ and $b = n$, we have $\binom{n + g - 1}{g}$ constraints of type \eqref{mDOMP_sB2:c1}. Now, the claim follows from the well-known identity $\binom{N}{k} = \binom{N}{N-k}$ for $0 \leq k \leq N$.
    \end{proof}
\end{lemma}

Recall from Section~\ref{sec:transportation} that the dual solutions given by Theorem~\ref{theorem_u_v_formulas} and Theorem~\ref{theorem_u_v_formulas_alternative} coincide up to a shift by $\lambda_1 c_{(0)} = 0$ (cf. Lemma~\ref{lemma_initialization}). So, in fact, the two formulas yield identical values in our application. This, together with Lemma~\ref{lemma_num_constraints} is a strong indicator that the models $\text{mDOMP}_{sB1}$ and $\text{mDOMP}_{sB2}$ are actually identical. Indeed, this is what we show in the next lemma.
\begin{lemma}
    The constraint sets \eqref{mDOMP_sB1:c1} and \eqref{mDOMP_sB2:c1} are identical.
    \begin{proof}
        Let $f^1: [n] \rightarrow [g]_0$, $f(1) = 0$ be a monotonically increasing function. This function defines the constraint of type \eqref{mDOMP_sB1:c1} of the form $\theta \geq \sum_{l = 1}^n u_l^{f^1} + \sum_{h = 0}^g v_h^{f^1} \sum_{\substack{i,j \in [n] \\ c_{ij} = c_{(h)}}} x_{ij}$, where $u^{f^1}$ and $v^{f^1}$ are defined according to \eqref{u_f_v_f_formulas}. We can associate a unique staircase path on the $n \times (g+1)$ grid that connects the northwest with the southeast corner: consider the path $\mathcal{A}$ from $(1, 0)$ to $(n, g)$ that contains all points $(l, f^1(l))$ for $l \in [n]$ such that $f^1(l) = \min\{h \in [g]_0 \, | \, (l,h) \in \mathcal{A} \}$. This path has the form
        \begin{align*}
            \mathcal{A} = \, &\big\{\big(1,0\big), \dots, \big(1, f^1(2)\big), \big(2, f^1(2)\big), \dots, \big(2, f^1(3)\big), \big(3, f^1(3)\big), \dots \\
            & \dots, \big(n-1, f^1(n)\big), \big(n, f^1(n)\big), \dots, \big(n, g\big)\big\}.
        \end{align*}
        Define $f^2:[g-1]_0 \rightarrow [n]$ as it was done earlier in Section~\ref{sec:transportation} by $f^2(h) := \max \{l \in [n] \, | \, (l,h) \in \mathcal{A}\}$ for every $h \in [g-1]_0$. By the staircase structure of $\mathcal{A}$, we see that $f^2$ is monotonically increasing. Therefore, it defines the constraint of type \eqref{mDOMP_sB2:c1} of the form $\theta \geq \sum_{l = 1}^n u_l^{f^2} + \sum_{h = 0}^g v_h^{f^2} \sum_{\substack{i,j \in [n] \\ c_{ij} = c_{(h)}}} x_{ij}$, where $u^{f^2}$ and $v^{f^2}$ are defined according to \eqref{u_f_v_f_formulas_alternative}. Lemma~\ref{lemma_initialization} implies that $(u^{f^2}, v^{f^2}) = (u^{f^1} + \lambda_1 c_{(0)}, v^{f^1} - \lambda_1 c_{(0)}) = (u^{f^1}, v^{f^1})$. As a result, the constraints induced by $f^1$ and $f^2$ are identical.

        Conversely, let $f^2: [g-1]_0 \rightarrow [n]$ be a monotonically increasing function. Again, this function uniquely describes a staircase path connecting the northwest and the southeast corner, namely the path $\mathcal{A}$ from $(1, 0)$ to $(n, g)$ containing all points $(f^2(h), h)$ for $h \in [g-1]_0$ such that $f^2(h) = \max\{l \in [n] \, | \, (l,h) \in \mathcal{A} \}$. It has the form
        \begin{align*}
            \mathcal{A} = \, &\big\{\big(1,0\big), \dots, \big(f^2(0), 0\big), \big(f^2(0), 1\big), \dots, \big(f^2(1), 1\big), \big(f^2(1), 2\big), \dots \\
            &\dots, \big(f^2(g-1), g-1\big), \big(f^2(g-1), g\big), \dots, \big(n, g\big)\big\}.
        \end{align*}
        For $f^1:[n] \rightarrow [g]_0$ with $f^1(1) := 0$ and $f^1(l) := \min \{h \in [g]_0 \, | \, (l,h) \in \mathcal{A}\}$ for every $l \in \{2, \dots, n\}$ we again find that $f^2$ and $f^1$ induce the exact same constraint.
    \end{proof}
\end{lemma}

Even though $\text{mDOMP}_{sB1}$ and $\text{mDOMP}_{sB2}$ describe the same polytope, they differ from a computational perspective, which can make one formulation preferable to the other. This issue is discussed in Section~\ref{sec:practical_notes} and revisited in Section~\ref{sec:discussion}.

\subsection{Practical notes}\label{sec:practical_notes}
Since there are exponentially many constraints of type \eqref{mDOMP_sB1:c1} and \eqref{mDOMP_sB2:c1}, they need to be separated when solving the models $\text{mDOMP}_{sB1}$ and $\text{mDOMP}_{sB2}$ in practice. Given an integer solution $(\bar{x}, \bar{y}) \in X_I$ at some node of the branch-and-bound (B\&B) tree, one needs to identify the corresponding monotonically increasing function $f$ that encodes the staircase path traversed by the northwest-corner rule. If the objective function value $\bar{\theta}$ at that node is less than $\sum_{l = 1}^n u^f_l + \allowbreak \sum_{h = 0}^g v^f_h  \allowbreak \sum_{\substack{i,j \in [n] \\ c_{ij} = c_{(h)}}} \bar{x}_{ij}$ minus a cut violation tolerance $\varepsilon \geq 0$, then the constraint ${\bar{\theta}\ge } \sum_{l = 1}^n u^f_l + \allowbreak \sum_{h = 0}^g v^f_h  \allowbreak \sum_{\substack{i,j \in [n] \\ c_{ij} = c_{(h)}}} x_{ij}$ is added to the model. Algorithm~\ref{cut_separation_algo} summarizes this procedure.

\begin{algorithm}[ht]
\caption{Benders cut separation for $\text{mDOMP}_{sB1}$ and $\text{mDOMP}_{sB2}$}\label{cut_separation_algo}
\DontPrintSemicolon
Consider an integer solution $(\bar{x}, \bar{y}) \in X_I$ with objective value $\bar{\theta}$ at some B\&B node\;
Identify the monotonically increasing $f$ corresponding to $(\bar{x}, \bar{y})$ \;
Compute the vectors $u^f$ and $v^f$ \;
Set $\mathrm{RHS} \leftarrow \sum_{l = 1}^n u^f_l + \sum_{h = 0}^g v^f_h \sum_{\substack{i,j \in [n] \\ c_{ij} = c_{(h)}}} \bar{x}_{ij}$\;
\If{$\bar{\theta} < \mathrm{RHS} - \varepsilon$}{
    Add the cut $\theta \geq \sum_{l = 1}^n u^f_l + \sum_{h = 0}^g v^f_h \sum_{\substack{i,j \in [n] \\ c_{ij} = c_{(h)}}} x_{ij}$ \;
}
\end{algorithm}

Next, we describe how to obtain the function $f$ from a given solution $(\bar{x}, \bar{y}) \in X_I$ for both cases $\text{mDOMP}_{sB1}$ and $\text{mDOMP}_{sB2}$. The first step is to find the indices $h_l$ or $l_h$, respectively, that encode the solution produced by Algorithm~\ref{NW_corner_rule} on subproblem \eqref{staircase_sub}. These indices are given by Corollary~\ref{corollary_get_indices}.
\begin{enumerate}
    \item[\textit{(i)}] For $l \in \{2, \dots, n\}$, $h_l = \min \{ h \in [g]_0 \, | \, l-1 \leq \sum_{k = 0}^{h} \bar{x}_k \}$. 
    \item[\textit{(ii)}] For $h \in \{0, \dots, g-1\}$, $l_h = \max \{ l \in [n] \, | \, l-1 \leq \sum_{k = 0}^{h} \bar{x}_k \} = \min \{ 1 + \sum_{k = 0}^{h} \bar{x}_k, n \}$.
\end{enumerate}
Then, one may set $f(1) = 0$, $f(l) = h_l$ for $l \in \{2, \dots, n\}$ or $f(h) = l_h$ for $h \in \{0, \dots, g-1\}$, respectively.

As mentioned in Section~\ref{sec:theoretical_notes}, it does not matter theoretically which formulas are used to compute $u$ and $v$. However, according to \textit{(i)} and \textit{(ii)} there are obvious computational differences when calculating the indices $h_l$ and $l_h$. Moreover, the sums in the formulas given by Theorem~\ref{theorem_u_v_formulas} and Theorem~\ref{theorem_u_v_formulas_alternative} may differ significantly in their number of terms, which also affects the efficiency of a numerical evaluation. For these reasons, we tested both variants in our numerical study.

Notice that solving the transportation subproblem directly with the northwest-corner rule and retrieving the duals using Algorithm~\ref{dual_algo_backward} or \ref{dual_algo_forward} results in a time-complexity of $\mathcal{O}(n+g)$, while any straightforward implementation of Algorithm~\ref{cut_separation_algo} runs in $\mathcal{O}(ng)$. Nevertheless, as most of the computational time in our experiments is spent solving LP relaxations within the B\&B process rather than performing the separation procedure, this difference has only a minor impact on the overall runtime. To highlight the practical differences between $\text{mDOMP}_{sB1}$ and $\text{mDOMP}_{sB2}$, we therefore use Algorithm~\ref{cut_separation_algo} in the numerical experiments.

Lastly, for a proper initialization, one must add a generally valid lower bound on $\theta$ to each model. To obtain such a bound, compute $c^{\min}_i = \min\{c_{ij} \, | \, j \in [n]\}$ and $c^{\max}_i = \max\{c_{ij} \, | \, j \in [n]\}$ for $i \in [n]$ and sort $c^{\min}$ and $c^{\max}$ non-decreasingly. Denoting the sorted vectors by $c^{\min}_{(\cdot)}$ and $c^{\max}_{(\cdot)}$, it is easy to see that the objective function value always satisfies 
\begin{equation*}
    \theta \geq \sum_{i \in [n]: \, \lambda_{i} \geq 0} \lambda_i c^{\min}_{(i)} + \sum_{i \in [n]: \, \lambda_{i} < 0} \lambda_i c^{\max}_{(i)}.
\end{equation*}
This bound is independent of any model and therefore works for both \eqref{mDOMP_sB1} and \eqref{mDOMP_sB2}.

\section{Numerical experiments}\label{sec:numerical_experiments}
In this section, we present the results of our numerical experiments which we conducted to test $\text{mDOMP}_{sB1}$ and $\text{mDOMP}_{sB2}$. The code was written in Python~3.12 and Gurobi~12.0.3 \citep{gurobi} was chosen as our MILP solver. For an unbiased comparison of the models, we disabled presolve, heuristics, as well as solver generated cuts, and we restricted the number of threads to one, following the settings in \cite{ljubic_benders_2024}. Regarding the termination criteria, a time limit of two hours (7200 seconds) and a MIP Gap of $10^{-6}$ per solve were set. Detailed information on these criteria can be found in Gurobi's reference manual \citep{gurobi}. Next, we implemented the necessary separation procedures (Algorithm~\ref{cut_separation_algo}) as described in Section~\ref{sec:practical_notes} by means of Gurobi's callback mechanism. The cut violation parameter was chosen as $\varepsilon = 0.01$, again complying with the choice in \cite{ljubic_benders_2024}. All computations were performed on AMD EPYC 9454 processors on the bwUniCluster system.

In the remainder of this section, we describe our test instances (Section~\ref{sec:instances}), briefly introduce two state-of-the-art models against which we compare our new formulations (Section~\ref{sec:reference_models}), and discuss our findings (Section~\ref{sec:discussion}).

\subsection{Instances}\label{sec:instances}
We are testing our DOMP formulations on two sets of instances: a collection of randomly generated instances and a benchmark set taken from the literature. For the randomly generated set, the number of locations $n$ takes the values $20$, $30$, $50$, $100$, $150$, and $200$. For each value of $n$, five $n \times n$ cost matrices $c$ with random entries in $[100, 1000]$ and two decimal places are generated. The second set is based on the well-known $p$-median instances from OR-Library \citep{Beasley_original_dataset} and will henceforth also be referred to as Beasley instances. Specifically, we use the cost matrices of sizes $n = 100$, $200$, and $300$ (five matrices per $n$, corresponding to \texttt{pmed1}, \dots, \texttt{pmed15}). Because these matrices represent shortest-path distances, they satisfy the triangle inequality. Moreover, the Beasley cost matrices are symmetric, contain many repeated entries, and have a zero diagonal, which is known as the free self-service property in the location literature. For both datasets, the number $p$ of locations to be opened is set to $\lfloor \frac{n}{10} \rfloor$, $\lfloor \frac{n}{4} \rfloor$, $\lfloor \frac{n}{3} \rfloor$, or $\lfloor \frac{n}{2} \rfloor$.

The considered objective weight vectors $\lambda$ are all non-increasing, as required by Theorems~\ref{transport_domp_formulation} and \ref{transport_domp_formulation_alternative}. They include common obnoxious facility criteria such as the obnoxious $p$-median ($\lambda_1$), $p$-center ($\lambda_2$), $k$-centrum ($\lambda_3$), and range ($\lambda_6$) objectives. Also, the reverse criterion ($\lambda_7$) has appeared earlier in the DOMP literature for numerical testing \citep[e.g.,][]{dominguez_discrete_2020, ljubic_benders_2024}. Additionally, we consider three further, less studied objectives ($\lambda_4, \lambda_5, \lambda_8$) to obtain a more comprehensive picture of the models' performance. The ninth type of objective vector ($\lambda_9$) is obtained by randomly generating $n$-vectors with entries in $[-n, n] \cap \mathbb{Z}$ and sorting them in non-increasing sequence. Despite their limited interpretability, random objectives provide a useful benchmark, as their unstructured nature is expected to yield instances that are difficult to solve. Finally, as a canonical baseline from location theory, we also evaluate the standard $p$-median objective ($\lambda_{10}$). Table~\ref{tab:lambdas} provides a summary of all tested $\lambda$-vectors.

\begin{table}[H]
    \centering
    \scriptsize
    \begin{tabular}{lll}
        \hline
        \textbf{Notation} & \textbf{$\lambda$-vector} & \textbf{Name} \\ \hline
        $\lambda_1$ & $(-1, \dots, -1)$ & obnoxious $p$-median \\ 
        $\lambda_2$ & $(0, \dots, 0, -1)$ & obnoxious $p$-center \\ 
        $\lambda_3$ & $(0, \dots, 0, \underbrace{-1, \dots, -1}_k)$ & obnoxious $k$-centrum, $k = \lfloor \frac{n}{2} \rfloor$ \\
        $\lambda_4$ & $(\underbrace{1, \dots, 1}_k, 0, \dots, 0)$ & $k$-min, $k = \lfloor \frac{n}{2} \rfloor$ \\
        $\lambda_5$ & $(\underbrace{1, \dots, 1}_k, -1, \dots, -1)$ & obnoxious $k$-range, $k = \lfloor \frac{n}{2} \rfloor$ \\
        $\lambda_6$ & $(1, 0, \dots, 0, -1)$ & obnoxious range \\
        $\lambda_7$ & $(n, n-1, \dots, 2, 1)$ & reverse \\
        $\lambda_8$ & $(-1, -2, \dots, -(n-1), -n)$ & negative reverse \\
        $\lambda_9$ & $\lambda \in \{-n, -n+1, \dots, n-1, n\}^n$ & random non-increasing \\
        $\lambda_{10}$ & $(1, \dots, 1)$ & $p$-median \\
        \hline
    \end{tabular}
    \caption{Summary of tested $\lambda$-vectors.}
    \label{tab:lambdas}
\end{table}

Every tuple $(n, p, c, \lambda)$ forms a single test instance. For the randomly generated dataset, this results in $6 \cdot 4 \cdot 5 \cdot 10 = 1200$ test instances. Additionally, there are $3 \cdot 4 \cdot 5 \cdot 10 = 600$ Beasley instances. All $1800$ total instances are publicly available online \citep{instances2026dataset}.

\subsection{Reference models}\label{sec:reference_models}
We selected two benchmark models from the literature for comparison with our new formulations. Both were proposed by \citet{ljubic_benders_2024}.

The first one, referred to as $\text{mDOMP}_{rB}$ ($\text{DOMP}_{r0B}$ in \citet{ljubic_benders_2024}), is a Benders reformulation of the well-known radius formulation introduced by \citet{puerto_new_2008} and \citet{marin_flexible_2009}, a DOMP model that is similar to the staircase formulation \eqref{DOMP_s}. Although $\text{mDOMP}_{rB}$ has a polynomial number of constraints, it is advantageous to implement the separation procedure given by \citet{ljubic_benders_2024} in order to improve the model's efficiency. A comparison with $\text{mDOMP}_{rB}$ is particularly interesting, since it is a specialized model for non-increasing $\lambda$-vectors, similar to $\text{mDOMP}_{sB1}$ and $\text{mDOMP}_{sB2}$. Furthermore, because the size of the formulation depends on the number of distinct entries in the cost matrix, $\text{mDOMP}_{rB}$ is expected to perform well on datasets with many repeated cost values, such as the Beasley instances.

The second reference model, which we call $\text{DOMP}_{\theta B}$ ($\text{DOMP}_{\theta 1}$ in \cite{ljubic_benders_2024}), is a Benders reformulation of a compact DOMP formulation by \citet{marin_fresh_2020}. Like $\text{mDOMP}_{rB}$, this model has a polynomial number of constraints, but dynamic constraint separation is again used to improve its computational efficiency. $\text{DOMP}_{\theta B}$ turned out to be one of the best-performing DOMP models in \citet{ljubic_benders_2024} for random instances. Note that $\text{DOMP}_{\theta B}$ is valid for any $\lambda \in \mathbb{R}^n$ and does not require $\lambda$ to be non-increasing.

For more details on $\text{mDOMP}_{rB}$ and $\text{DOMP}_{\theta B}$, we refer the reader to \citet{ljubic_benders_2024}.

\subsection{Discussion}\label{sec:discussion}
In this subsection, we present and discuss some of the most relevant aspects for evaluating and comparing the performance of our proposed models. Specifically, we look at the number of solved instances alongside runtimes (Section~\ref{sec:opt_time}), optimality gaps and bounds (Section~\ref{sec:gaps_bounds}), as well as the efficiency of the models' separation procedures (Section~\ref{sec:separation}). The key findings are summarized in Section~\ref{sec:summary}. Detailed numerical results can be found online \citep{instances2026dataset}.

\subsubsection{Solved instances and runtime}\label{sec:opt_time}
To begin with, Figure~\ref{fig:random_opt_vs_time} summarizes, for each model, the number of instances from the random dataset solved to optimality and the corresponding average runtime, separated by instance size.
\begin{figure}[ht]
    \centering
    \includegraphics[width=0.95\linewidth]{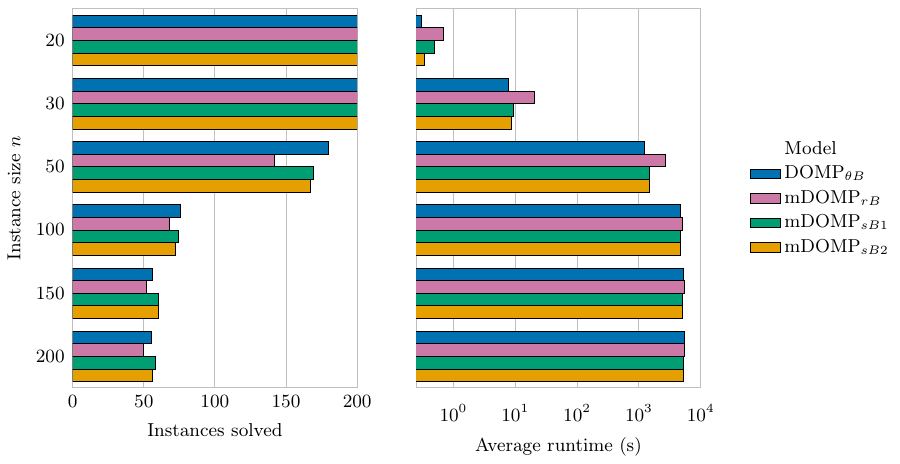}
    \caption{Number of random instances solved to optimality vs. average runtime per instance size $n$.}
    \label{fig:random_opt_vs_time}
\end{figure}
Unsurprisingly, the number of solved instances decreases for each model as $n$ increases, whereas the average runtimes increase with $n$. For small sizes ($n \in \{20, 30\}$), Gurobi solved all instances within the time limit across all models. At $n = 30$, the average runtimes are below $10$ seconds, except for $\text{mDOMP}_{rB}$ with roughly $20$ seconds. For medium-sized instances, differences between models become slightly more pronounced. At $n = 50$, $\text{mDOMP}_{rB}$ solved the fewest instances and therefore reached the time limit most frequently, resulting in the highest average runtime. In contrast, $\text{DOMP}_{\theta B}$ solved the most instances of this size, achieving the lowest average runtime. For larger instances ($n \in \{100, 150, 200\}$), only objectives $\lambda_4$, $\lambda_7$, and $\lambda_{10}$ are widely solved to optimality, while most other instances reach the time limit. As a result, average runtimes become similar across all models. However, $\text{mDOMP}_{rB}$ again solves fewer of these large instances to optimality than the other models. 
In line with the theoretical results, models $\text{mDOMP}_{sB1}$ and $\text{mDOMP}_{sB2}$ achieve similar performance across all instance sizes.

Turning to the instances from the literature, Figure~\ref{fig:Beasley_opt_vs_time} shows the number of solved instances and the average runtimes on the Beasley dataset.
\begin{figure}[ht]
    \centering
    \includegraphics[width=0.95\linewidth]{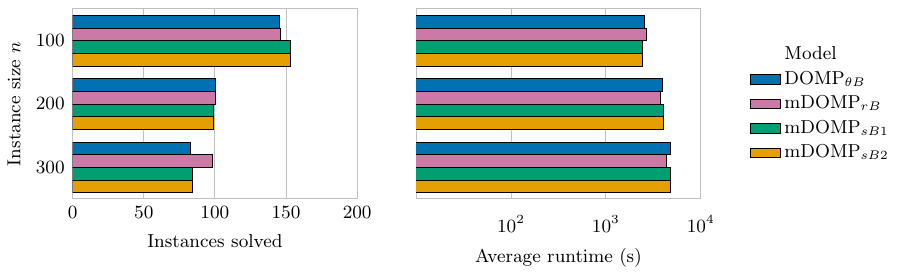}
    \caption{Number of Beasley instances solved to optimality vs. average runtime per instance size $n$.}
    \label{fig:Beasley_opt_vs_time}
\end{figure}
For sizes $n = 100$ and $n = 200$, all four formulations remain competitive. At $n = 100$, $\text{mDOMP}_{sB1}$ and $\text{mDOMP}_{sB2}$ stand out by solving the highest number of instances, accompanied by slightly lower average runtimes compared to the reference models. At $n = 200$, all models solve roughly half of the instances within similar average times. A clear differentiation emerges at the largest size, $n = 300$. Here, $\text{mDOMP}_{rB}$ outperforms the other formulations, still solving nearly half ($49$\%) of the instances of this size to optimality and recording the lowest average runtime. The remaining models solve approximately $42$\% for these instances.

Considering the aggregated numbers of solved instances per dataset given in Table~\ref{tab:solved_random_vs_beasley}, it is evident that the structured Beasley instances are computationally easier than the random instances.
\begin{table}[ht]
    \centering
    \begin{tabular}{lcccc}
        \hline
        Dataset & $\text{DOMP}_{\theta B}$ & $\text{mDOMP}_{rB}$ & $\text{mDOMP}_{sB1}$ & $\text{mDOMP}_{sB2}$ \\
        \hline
        Random  & 131 & 118 & 132 & 128 \\
        Beasley & 245 & 246 & 252 & 252 \\
        \hline
    \end{tabular}
    \caption{Total number of instances solved to optimality for $n \in \{100, 200\}$ (out of 400 instances per dataset: 200 for $n = 100$ and 200 for $n = 200$).}
    \label{tab:solved_random_vs_beasley}
\end{table}
More precisely, for $n \in \{100, 200\}$, all models find optimal solutions to more than $60\%$ of the Beasley instances, whereas they solve less than one third of the random instances of these sizes to optimality.

To provide some insights on the impact of the $\lambda$ objective vector on the computational results, Figure~\ref{fig:opt_by_lamb} plots the number of instances solved to optimality, disaggregated by the $\lambda$ type.
\begin{figure}[ht]
    \centering
    \includegraphics[width=0.95\linewidth]{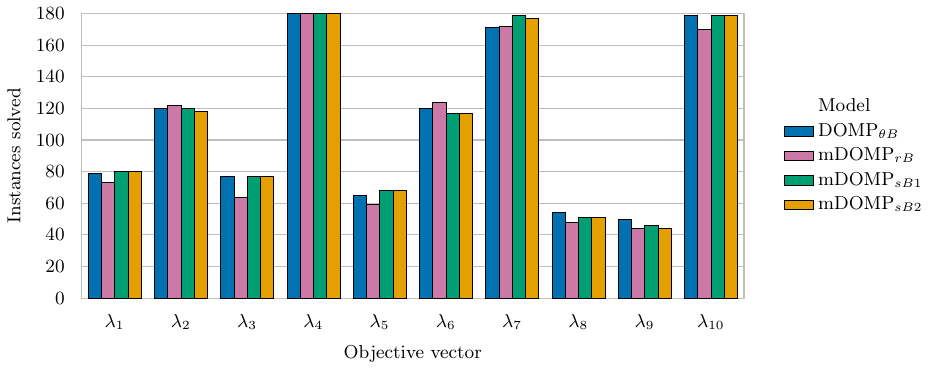}
    \caption{Total number of instances from all datasets solved to optimality per objective vector $\lambda$.}
    \label{fig:opt_by_lamb}
\end{figure}
This chart reveals that the computational difficulty of an instance depends on the evaluated criterion. Clearly, the $k$-min ($\lambda_4$), reverse ($\lambda_7$), and $p$-median ($\lambda_{10}$) objectives result in the easiest instances. Across all formulations, nearly all instances featuring these objectives were solved to optimality. For the $p$-median problem, this is known to be related to integer-friendliness, meaning that the LP relaxation frequently yields integer-valued solutions directly, and when it does not, only a small branching and bounding effort is needed to reach an optimum \citep[see][]{revelle_facility_1993}. Table~\ref{tab:bnb_by_lambda} confirms that for the $p$-median objective and shows the same pattern for $\lambda_4$ and $\lambda_7$: the average number of B\&B nodes required until the solver terminates with an optimal solution is lowest in these cases, regardless of the model.
\begin{table}[ht]
    \centering
    \begin{tabular}{lcccc}
        \hline
        $\lambda$ & $\text{DOMP}_{\theta B}$ & $\text{mDOMP}_{rB}$ & $\text{mDOMP}_{sB1}$ & $\text{mDOMP}_{sB2}$ \\
        \hline
        $\lambda_1$ & 58,887.44 & 52,135.03 & 74,551.41 & 71,548.36 \\
        $\lambda_2$ & 6,308.49 & 336.52 & 13,625.52 & 8,086.51 \\
        $\lambda_3$ & 112,147.38 & 57,370.42 & 120,664.29 & 122,662.60 \\
        $\lambda_4$ & 36.35 & 1.21 & 4.93 & 5.90 \\
        $\lambda_5$ & 115,622.65 & 36,171.24 & 117,471.82 & 123,134.82 \\
        $\lambda_6$ & 16,117.29 & 367.01 & 9,943.26 & 6,335.60 \\
        $\lambda_7$ & 37.58 & 20.54 & 99.82 & 63.80 \\
        $\lambda_8$ & 36,849.11 & 9,451.02 & 38,580.06 & 33,629.18 \\
        $\lambda_9$ & 35,978.84 & 10,410.18 & 39,407.07 & 38,776.57 \\
        $\lambda_{10}$ & 1,124.24 & 36.27 & 1,513.13 & 1,297.55 \\
        \hline
    \end{tabular}
    \caption{Average number of B\&B nodes over solved instances from all datasets for each model and $\lambda$-vector, rounded to two decimal places.}
    \label{tab:bnb_by_lambda}
\end{table}
Furthermore, $\lambda_2$ and $\lambda_6$ can be seen as medium difficult for the tested models, since roughly two thirds of them can be solved to optimality. The remaining obnoxious criteria ($\lambda_1, \lambda_3, \lambda_5$), alongside the negative reverse ($\lambda_8$) and random non-increasing ($\lambda_9$) objectives, pose the greatest computational challenge for all models. Certifying optimality requires high branching efforts and in most cases more than the two hour time limit. The comparison between the models shows that $\text{mDOMP}_{rB}$ falls behind on $\lambda_1$, $\lambda_3$, $\lambda_5$, $\lambda_8$, and $\lambda_{10}$, but outperforms the other models on $\lambda_2$ and $\lambda_6$ in terms of number of solved instances. Moreover, our models solve the most instances overall on $\lambda_5$ and $\lambda_7$, and $\text{DOMP}_{\theta B}$ is slightly ahead on $\lambda_8$ and $\lambda_9$.

Finally, Table~\ref{tab:runtime_by_p} confirms a known relationship between the number of opened facilities $p$ and overall computational difficulty.
\begin{table}[ht]
    \centering
    \begin{tabular}{ccccc}
        \hline
        $p$ & $\text{DOMP}_{\theta B}$ & $\text{mDOMP}_{rB}$ & $\text{mDOMP}_{sB1}$ & $\text{mDOMP}_{sB2}$ \\
        \hline
        $\lfloor \frac{n}{10} \rfloor$ & 3166.29 (173) & 3505.64 (206) & 3167.49 (176) & 3175.89 (180) \\
        $\lfloor \frac{n}{4} \rfloor$ & 3137.12 (184) & 3284.19 (187) & 3116.79 (179) & 3109.57 (180) \\
        $\lfloor \frac{n}{3} \rfloor$ & 3135.53 (181) & 3316.12 (189) & 3153.02 (179) & 3152.18 (180) \\
        $\lfloor \frac{n}{2} \rfloor$ & 2922.61 (167) & 2972.81 (162) & 2917.93 (169) & 2920.06 (169) \\
        \hline
    \end{tabular}
    \caption{Average runtime in seconds and count of unsolved instances in parentheses (out of all 450 instances per $p$) for different numbers of opened facilities $p$. Runtimes are rounded to two decimal places.}
    \label{tab:runtime_by_p}
\end{table}
Across all four formulations, the highest average runtimes occur at $p = \lfloor \frac{n}{10} \rfloor$, while the lowest are observed at $p = \lfloor \frac{n}{2} \rfloor$. For $p = \lfloor \frac{n}{4} \rfloor$ and $p = \lfloor \frac{n}{3} \rfloor$, the average runtimes are similar, with no clear ordering between the two values. Also, all four models have the lowest number of unsolved instances at $p = \lfloor \frac{n}{2} \rfloor$. This observation is somewhat unintuitive, since the number of candidate solutions is largest in this case. We do not have a sound explanation for this phenomenon at the moment, but this behavior is consistent with previous studies \citep{labbe_comparative_2017, marin_fresh_2020, ljubic_benders_2024}.

\subsubsection{Gaps and bounds}\label{sec:gaps_bounds}
Next, we analyze relative optimality gaps at termination to assess solution quality in more detail, especially for non-optimal runs. Let $\overline{z}$ denote the objective value of the incumbent solution and $\underline{z}$ the best lower bound obtained during the optimization. Since these values may have different signs in our setting, we use the following symmetric normalization to measure relative optimality gaps:
\[\text{gap} := \frac{\overline{z} - \underline{z}}{\max\{1, |\overline{z}|, |\underline{z}|\}} \in [0, 2].\] 
If no solution has been found, the gap is set to the formula's maximum possible value of $2$.
Figure~\ref{fig:random_cumulative_gaps} shows, for each model, the cumulative fraction of instances (y-axis) achieving a given gap value (x-axis) on the random instances.
\begin{figure}[ht]
    \centering
    \includegraphics[width=0.85\linewidth]{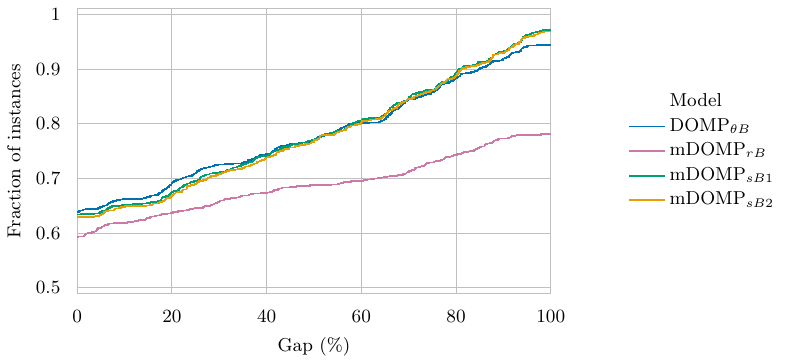}
    \caption{Cumulative distribution of optimality gaps on random instances at termination.}
    \label{fig:random_cumulative_gaps}
\end{figure}
The plot clearly indicates that $\text{mDOMP}_{rB}$ is the least robust model on this dataset, in the sense that it solves fewer instances than all other models within any given gap value. The other three models are more competitive and $\text{DOMP}_{\theta B}$ holds a slight advantage for small gaps up to roughly $35$\%. $\text{mDOMP}_{sB1}$ and $\text{mDOMP}_{sB2}$ are the most successful in producing any incumbent solution (approximately $97$\%). The cumulative gap curves of the latter two models behave similarly, with $\text{mDOMP}_{sB1}$ reaching slightly better gaps than $\text{mDOMP}_{sB2}$.

To understand the drivers behind these gaps, Table~\ref{tab:random_best_bounds} reports the number of times each model achieved the best lower or upper bound on the random dataset.
\begin{table}[ht]
    \centering
    \begin{tabular}{lcccc}
        \hline
        & $\text{DOMP}_{\theta B}$ & $\text{mDOMP}_{rB}$ & $\text{mDOMP}_{sB1}$ & $\text{mDOMP}_{sB2}$ \\
        \hline
        Best lower bound & 927 & 845 & 916 & 938 \\
        Best upper bound & 902 & 730 & 902 & 881 \\
        \hline
    \end{tabular}
    \caption{Number of random instances where each model achieved the best lower and upper bounds within a relative tolerance of $10^{-4}$ compared to all other models.}
    \label{tab:random_best_bounds}
\end{table}
These numbers confirm the weakness of $\text{mDOMP}_{rB}$: it obtains both the least best lower bounds and upper bounds. In contrast, our proposed models and $\text{DOMP}_{\theta B}$ consistently find higher quality bounds. An interesting observation is that $\text{mDOMP}_{sB1}$ finds better upper bounds than $\text{mDOMP}_{sB2}$, but weaker lower bounds, largely equalizing the overall gap performance.

Next, we turn to the Beasley dataset. Figure~\ref{fig:beasley_cumulative_gaps} shows the cumulative gap distribution for these structured instances.
\begin{figure}[ht]
    \centering
    \includegraphics[width=0.85\linewidth]{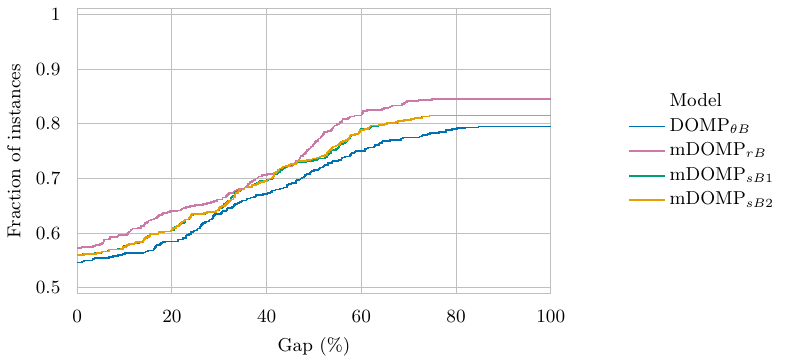}
    \caption{Cumulative distribution of optimality gaps on Beasley instances at termination.}
    \label{fig:beasley_cumulative_gaps}
\end{figure}
In contrast to the random instances, $\text{mDOMP}_{rB}$ dominates the gap comparison in this case. It achieves gaps below 20\% in roughly 65\% of the instances and flattens out at an 85\% incumbent rate. Conversely, $\text{DOMP}_{\theta B}$ is the weakest model here. Our formulations, $\text{mDOMP}_{sB1}$ and $\text{mDOMP}_{sB2}$, sit in the middle, demonstrating robustness across both data types. Table~\ref{tab:beasley_best_bounds} confirms the reversed gap performance pattern: $\text{mDOMP}_{rB}$ benefits from the repeated cost structure of the Beasley dataset to find the best lower bound in $482$ instances and the best upper bound in $434$, significantly outperforming all other formulations.
\begin{table}[ht]
    \centering
    \begin{tabular}{lcccc}
        \hline
        & $\text{DOMP}_{\theta B}$ & $\text{mDOMP}_{rB}$ & $\text{mDOMP}_{sB1}$ & $\text{mDOMP}_{sB2}$ \\
        \hline
        Best lower bound & 361 & 482 & 377 & 390 \\
        Best upper bound & 368 & 434 & 379 & 377 \\
        \hline
    \end{tabular}
    \caption{Number of Beasley instances where each model achieved the best lower and upper bounds within a relative tolerance of $10^{-4}$ compared to all other models.}
    \label{tab:beasley_best_bounds}
\end{table}

Finally, Table~\ref{tab:gaps_by_p} details the average optimality gaps grouped by the number of opened facilities $p$.
\begin{table}[ht]
    \centering
    \begin{tabular}{ccccc}
        \hline
        $p$ & $\text{DOMP}_{\theta B}$ & $\text{mDOMP}_{rB}$ & $\text{mDOMP}_{sB1}$ & $\text{mDOMP}_{sB2}$ \\
        \hline
        $\lfloor \frac{n}{10} \rfloor$ & 38.92\% & 51.04\% & 41.23\% & 42.46\% \\
        $\lfloor \frac{n}{4} \rfloor$ & 34.77\% & 48.34\% & 30.55\% & 30.92\% \\
        $\lfloor \frac{n}{3} \rfloor$ & 34.41\% & 48.97\% & 29.15\% & 29.12\% \\
        $\lfloor \frac{n}{2} \rfloor$ & 32.84\% & 47.98\% & 27.07\% & 27.08\% \\
        \hline
    \end{tabular}
    \caption{Average gaps over all instances from both datasets for different numbers of opened facilities $p$, rounded to two decimal places.}
    \label{tab:gaps_by_p}
\end{table}
Consistent with the runtime and timeout metrics discussed in Section~\ref{sec:opt_time}, instances with $p = \lfloor \frac{n}{10} \rfloor$ yield the largest average gaps across all formulations. As $p$ increases to $\lfloor \frac{n}{2} \rfloor$, the average gaps are significantly reduced and there are no big differences between $p = \lfloor \frac{n}{3} \rfloor$ and $p = \lfloor \frac{n}{4} \rfloor$ for each model. Notably, our models $\text{mDOMP}_{sB1}$ and $\text{mDOMP}_{sB2}$ provide the tightest average gaps for medium to large $p$ values ($p \geq \lfloor \frac{n}{4} \rfloor$), outperforming both reference models.

\subsubsection{Separation procedures}\label{sec:separation}
Finally, we have a closer look at the efficiency of the separation procedures. Particularly, we report in Figure~\ref{fig:random_callback}, for each model, the average time spent in callbacks and the average number of added lazy constraints per random instance.
\begin{figure}[ht]
    \centering
    \includegraphics[width=0.95\linewidth]{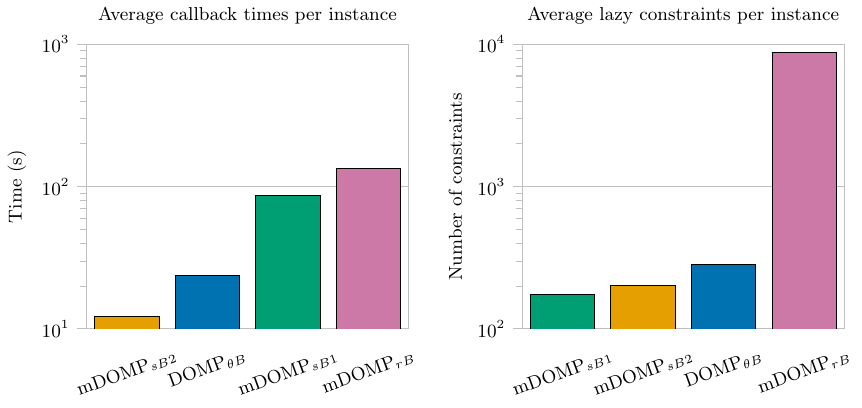}
    \caption{Average callback times and numbers of added lazy constraints across all random instances.}
    \label{fig:random_callback}
\end{figure}
Since $\text{mDOMP}_{rB}$ adds by far the largest number of lazy constraints, it also spends, on average, more time in the callback than the other models. $\text{DOMP}_{\theta B}$ generates roughly 1.5 times as many lazy constraints as $\text{mDOMP}_{sB1}$ and $\text{mDOMP}_{sB2}$, yet its average callback time lies between the two, lower than for $\text{mDOMP}_{sB1}$ but higher than for $\text{mDOMP}_{sB2}$. Notably, $\text{mDOMP}_{sB1}$ and $\text{mDOMP}_{sB2}$ achieve gaps and bounds competitive with the benchmark model $\text{DOMP}_{\theta B}$ on the random instances (see Figure~\ref{fig:random_cumulative_gaps} and Table~\ref{tab:random_best_bounds}) while adding fewer lazy constraints, indicating that their cuts are rather strong and effective. Further, the large number of comparatively weaker cuts generated by $\text{mDOMP}_{rB}$ contributes to its inferior performance on this dataset.

We further comment on the observed differences in callback times between $\text{mDOMP}_{sB1}$ and $\text{mDOMP}_{sB2}$, which stem from the distinct implementations of Algorithm~\ref{cut_separation_algo}. For $\text{mDOMP}_{sB1}$, the computational bottleneck is the evaluation of \eqref{u_f_v_f_formulas} to determine the $v^f$-variables. This requires a for-loop over all $h \in [g]$ with an inner loop over $k \in [n]$ while $f(k) < h$. In contrast, $\text{mDOMP}_{sB2}$ computes the $u^f$-values via \eqref{u_f_v_f_formulas_alternative}, using a for-loop over $l \in [n]$ and an inner loop over $k \in [g-1]_0$ while $f(k) < l$. In the worst case, both implementations perform $ng$ inner iterations. As the cost matrices are randomly drawn from a sufficiently large set of values, we typically have $g \gg n$. In $\text{mDOMP}_{sB1}$, the outer loop iterates over all $g$ indices independently of the data, leading to consistently high iteration counts. On the other hand, $\text{mDOMP}_{sB2}$ places the larger index set inside the while loop, where the termination criterion usually truncates the iterations early. Consequently, the effective number of iterations is substantially smaller for $\text{mDOMP}_{sB2}$, explaining its faster average callback times. Note that under different structural conditions, in particular when $g < n$, this effect may reverse.

On the Beasley instances, where $g \not\gg n$, the differences between $\text{mDOMP}_{sB1}$ and $\text{mDOMP}_{sB2}$ in their callback behavior are less significant as Figure~\ref{fig:beasley_callback} shows.
\begin{figure}[ht]
    \centering
    \includegraphics[width=0.95\linewidth]{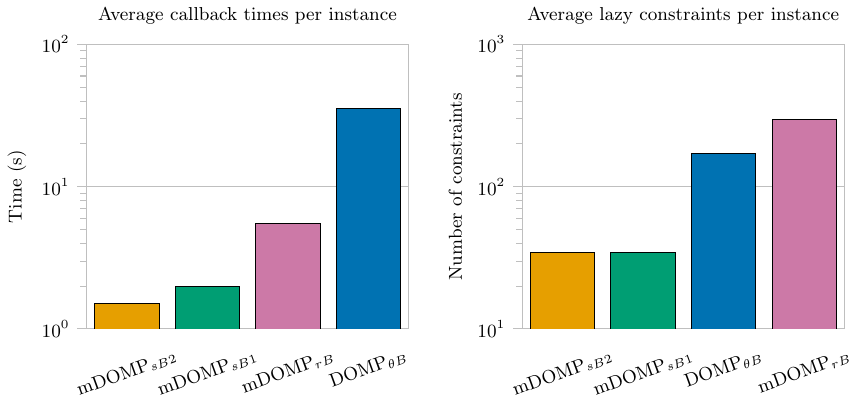}
    \caption{Average callback times and numbers of added lazy constraints across all Beasley instances.}
    \label{fig:beasley_callback}
\end{figure}
$\text{DOMP}_{\theta B}$ requires the highest average callback time on these structured instances, and its cuts appear less effective given its weaker performance noted in the previous subsection. 
While $\text{mDOMP}_{rB}$ still generates the highest volume of lazy constraints, the repeated cost entries in the Beasley matrices have the effect that the model needs fewer such constraints to sufficiently restrict the search space and arrive at good solutions, resulting in its observed dominance on this dataset.

\subsubsection{Summary}\label{sec:summary}
The computational results reveal dataset-dependent strengths of the two benchmark formulations $\text{mDOMP}_{rB}$ and $\text{DOMP}_{\theta B}$ and show how our proposed models perform relative to them. $\text{mDOMP}_{rB}$ strongly benefits from repeated entries in the cost matrix. It dominates on the Beasley dataset, but performs weakest on randomly generated instances. For $\text{DOMP}_{\theta B}$, we observed the opposite: it excels on random data but falls behind on the Beasley instances. While $\text{mDOMP}_{sB1}$ and $\text{mDOMP}_{sB2}$ do not surpass the benchmark models on their strong dataset, they clearly outperform them on their weaker ones. This consistent performance makes our formulations a reliable, general-purpose choice across diverse cost structures and objective criteria.

Besides model comparisons, our extensive numerical study yields some insights into computational problem difficulty. In particular, all models solve roughly twice as many instances of sizes $n \in \{100, 200\}$ from the Beasley dataset as from the random dataset, suggesting that the structured costs make these instances easier for the solver. The objective criterion is another important factor. Integer-friendly criteria like the $p$-median require little branching, leading to fast solve times and consistent solvability up to the largest instances with $n = 300$ for all models. In contrast, obnoxious objectives ($\lambda_1, \lambda_3, \lambda_5$), alongside negative reverse ($\lambda_8$) and random non-increasing ($\lambda_9$) weights, lack this integer-friendliness, resulting in large B\&B trees and frequent timeouts. Also, the number of opened facilities $p$ influences the computational difficulty: all models struggled, on average, most with $p = \lfloor \frac{n}{10} \rfloor$ and least with $p = \lfloor \frac{n}{2} \rfloor$ in our numerical experiments.

Finally, when interpreting the absolute runtimes and gaps reported in this study, it is important to recall that our experimental design intends to isolate the theoretical strength of the formulations. Therefore, we disabled Gurobi’s presolve, heuristics, solver-generated cuts, and multi-threading. Moreover, we omitted further enhancements like tailored valid inequalities or variable fixing. Enabling solver features and model enhancements is expected to improve the practical scalability of all tested models.

\section{Conclusion and further research}
\label{sec:conclusion}
In this paper, we studied the TP under a Monge cost structure and showed how the optimal solution obtained by the northwest-corner rule can be translated into explicit solution formulas for the corresponding dual problem. Because the TP frequently arises, often as a substructure of more complex optimization models, such dual formulas can enable the application of advanced optimization techniques and provide deeper insight into the structure of these problems. To illustrate this potential, we employed the derived formulas to construct two Benders reformulations, $\text{mDOMP}_{sB1}$ and $\text{mDOMP}_{sB2}$, of a MILP formulation for the DOMP with a non-increasing objective weight vector.

Both formulations demonstrate strong computational performance and prove to be competitive with two state-of-the-art benchmark models from the literature. A particular strength lies in their robustness across a diverse range of instances with different cost structures and objective criteria. While the two benchmark models $\text{DOMP}_{\theta B}$ and $\text{mDOMP}_{rB}$ have clear data-specific strengths and weaknesses, our formulations' good performance is less dependent on a specific instance type. Although our new models describe the same polytope, the numerical experiments reveal that $\text{mDOMP}_{sB1}$ spends, on average, significantly more time in the separation callbacks than $\text{mDOMP}_{sB2}$ on instances with a large number of distinct cost entries. This highlights that one of the two models may be preferred over the other in practice. Furthermore, our experiments empirically verify the computational difficulty of obnoxious location criteria, contrasting them with classical and integer-friendly objectives like the $p$-median problem.

A promising direction for future research is the application of the derived formulas to other problems, some of which were mentioned in the introduction, where the TP under Monge costs arises naturally. Another avenue is the investigation of a possible generalization of the optimal dual formulas to the multidimensional TP with Monge costs, for which an optimal greedy algorithm is also known \citep[cf.][Theorem 6.5]{BurkardKlinzRudolf1996}. Finally, ongoing research aims to extend the models $\text{mDOMP}_{sB1}$ and $\text{mDOMP}_{sB2}$, currently restricted to non-increasing objective vectors $\lambda$, to arbitrary $\lambda$.

\section*{Acknowledgements}
This work was supported by the Deutsche Forschungsgemeinschaft (DFG, German Research Foundation) [project number 541767835] and the Agencia Estatal de Investigación (AEI, MICIU/AEI/10.13039/50110001 1033) [project numbers PCI2024-155024-2, PID2020-114594GB-C21, PID2024-156594NB-C21, CEX2024-001517-M].

The numerical experiments were performed on the computational resource bwUniCluster funded by the Ministry of Science, Research and the Arts Baden-Württemberg and the Universities of the State of Baden-Württemberg, Germany, within the framework program bwHPC.

\section*{Declaration of interest}
The authors declare that they have no known competing financial interests or personal relationships that could have appeared to influence the work reported in this paper. The funding agencies had no role in the design of the study; in the collection, analysis, or interpretation of data; in the writing of the manuscript; or in the decision to submit the article for publication.

\section*{Declaration of generative AI in the manuscript preparation process}
During the preparation of this work, the authors used ChatGPT (OpenAI) and Gemini (Google) to improve the readability and language of the manuscript and to assist in drafting selected code snippets. All AI-generated text and code were carefully reviewed, validated, and, where necessary, revised and adapted by the authors. The authors take full responsibility for the content of the published article.

\bibliography{references}

@book{nickel_puerto_2005,
	address = {Berlin/Heidelberg},
	title = {Location {Theory}},
	doi = {10.1007/3-540-27640-8},
	publisher = {Springer Berlin, Heidelberg},
	author = {Nickel, Stefan and Puerto, Justo},
	year = {2005}
}

@book{DomnguezMarin2003,
    title = {The {Discrete} {Ordered} {Median} {Problem}: {Models} and {Solution} {Methods}},
    doi = {10.1007/978-1-4419-8511-8},
    series = {Combinatorial Optimization},
    publisher = {Springer New York},
    author = {Domínguez-Marín, Patricia},
    year = {2003}
}

@article{boland_exact_2006,
	title = {Exact procedures for solving the discrete ordered median problem},
	volume = {33},
	doi = {10.1016/j.cor.2005.03.025},
	number = {11},
	journal = {Computers \& Operations Research},
	author = {Boland, Natashia and Domínguez-Marín, Patricia and Nickel, Stefan and Puerto, Justo},
	month = nov,
	year = {2006},
	pages = {3270--3300}
}

@article{fischetti_benders_2016,
	title = {Benders decomposition without separability: {A} computational study for capacitated facility location problems},
	volume = {253},
	doi = {10.1016/j.ejor.2016.03.002},
	number = {3},
	journal = {European Journal of Operational Research},
	author = {Fischetti, Matteo and Ljubić, Ivana and Sinnl, Markus},
	month = sep,
	year = {2016},
    pages = {557--569}
}

@article{dominguez-marin_heuristic_2005,
	title = {Heuristic {Procedures} for {Solving} the {Discrete} {Ordered} {Median} {Problem}},
	volume = {136},
	doi = {10.1007/s10479-005-2043-3},
	number = {1},
	journal = {Annals of Operations Research},
	author = {Domínguez-Marín, Patricia and Nickel, Stefan and Hansen, Pierre and Mladenović, Nenad},
	month = apr,
	year = {2005},
	pages = {145--173}
}

@article{nickel_velten_flexible_2017,
    title = {Optimization problems with flexible objectives: {A} general modeling approach and applications},
    volume = {258},
    doi = {10.1016/j.ejor.2016.08.045},
    number = {1},
    journal = {European Journal of Operational Research},
    author = {Nickel, Stefan and Velten, Sebastian},
    month = apr,
    year = {2017},
    pages = {79--88}
}

@article{labbe_comparative_2017,
	title = {A comparative study of formulations and solution methods for the discrete ordered p-median problem},
	volume = {78},
	doi = {10.1016/j.cor.2016.06.004},
	journal = {Computers \& Operations Research},
	author = {Labbé, Martine and Ponce, Diego and Puerto, Justo},
	month = feb,
	year = {2017},
	pages = {230--242}
}

@article{marin_flexible_2009,
	title = {A flexible model and efficient solution strategies for discrete location problems},
	volume = {157},
	doi = {10.1016/j.dam.2008.03.013},
	number = {5},
	journal = {Discrete Applied Mathematics},
	author = {Marín, Alfredo and Nickel, Stefan and Puerto, Justo and Velten, Sebastian},
	month = mar,
	year = {2009},
	pages = {1128--1145}
}

@incollection{nickel_discrete_2001,
	address = {Berlin, Heidelberg},
	title = {Discrete {Ordered} {Weber} {Problems}},
	doi = {10.1007/978-3-642-56656-1_12},
	booktitle = {Operations {Research} {Proceedings}},
	publisher = {Springer Berlin, Heidelberg},
    editor = {Fleischmann, Bernhard and Lasch, Rainer and Derigs, Ulrich and Domschke, Wolfgang and Rieder, Ulrich},
	author = {Nickel, Stefan},
	year = {2001},
	pages = {71--76}
}

@article{cherkesly_ranking_2025,
	title = {Ranking {Decomposition} for the {Discrete} {Ordered} {Median} {Problem}},
	volume = {37},
	doi = {10.1287/ijoc.2023.0059},
	number = {2},
	journal = {INFORMS Journal on Computing},
	author = {Cherkesly, Marilène and Contardo, Claudio and Gruson, Matthieu},
	month = mar,
	year = {2025},
	pages = {230--248}
}

@article{ljubic_benders_2024,
	title = {Benders decomposition for the discrete ordered median problem},
	volume = {317},
	doi = {10.1016/j.ejor.2024.04.030},
	number = {3},
	journal = {European Journal of Operational Research},
	author = {Ljubić, Ivana and Pozo, Miguel A. and Puerto, Justo and Torrejon, Alberto},
	month = sep,
	year = {2024},
	pages = {858--874}
}

@incollection{puerto_new_2008,
	address = {Berlin, Heidelberg},
	title = {A {New} {Formulation} of the {Capacitated} {Discrete} {Ordered} {Median} {Problems} with \{0, 1\}-{Assignment}},
	doi = {10.1007/978-3-540-77903-2_26},
	booktitle = {Operations {Research} {Proceedings} 2007},
	publisher = {Springer},
	author = {Puerto, Justo},
	editor = {Kalcsics, Jörg and Nickel, Stefan},
	year = {2008},
    pages = {165--170}
}

@article{marin_fresh_2020,
	title = {A fresh view on the {Discrete} {Ordered} {Median} {Problem} based on partial monotonicity},
	volume = {286},
	doi = {10.1016/j.ejor.2020.04.023},
	number = {3},
	journal = {European Journal of Operational Research},
	author = {Marín, Alfredo and Ponce, Diego and Puerto, Justo},
	month = nov,
	year = {2020},
	pages = {839--848}
}

@article{deleplanque_branch-price-and-cut_2020,
	title = {A {Branch}-{Price}-and-{Cut} {Procedure} for the {Discrete} {Ordered} {Median} {Problem}},
	volume = {32},
	doi = {10.1287/ijoc.2019.0915},
	number = {3},
	journal = {INFORMS Journal on Computing},
	author = {Deleplanque, Samuel and Labbé, Martine and Ponce, Diego and Puerto, Justo},
	month = jul,
	year = {2020},
	pages = {582--599}
}

@article{dominguez_discrete_2020,
	title = {Discrete ordered median problem with induced order},
	volume = {28},
	doi = {10.1007/s11750-020-00570-1},
	number = {3},
	journal = {TOP},
	author = {Domínguez, Enrique and Marín, Alfredo},
	month = oct,
	year = {2020},
	pages = {793--813}
}

@article{espejo_single_alloc_hub,
    title = {A new formulation and branch-and-cut method for single-allocation hub location problems},
    journal = {Computers \& Operations Research},
    volume = {155},
    pages = {Article 106241},
    year = {2023},
    doi = {10.1016/j.cor.2023.106241},
    author = {Espejo, Inmaculada and Marín, Alfredo and Muñoz-Ocaña, Juan M. and Rodríguez-Chía, Antonio M.}
}

@article{Pozo2021,
    title={The ordered median tree of hubs location problem},
    author={Pozo, Miguel A. and Puerto, Justo and Rodr{\'\i}guez-Ch{\'\i}a, Antonio M.},
    journal={TOP},
    volume={29},
    number={1},
    pages={78--105},
    year={2021},
    publisher={Springer},
    doi={10.1007/s11750-020-00572-z}
}

@article{Pozo2024,
    title = {The {O}rdered {M}edian {T}ree {L}ocation {P}roblem},
    journal = {Computers \& Operations Research},
    volume = {169},
    year = {2024},
    doi = {10.1016/j.cor.2024.106746},
    author = {Pozo, Miguel A. and Puerto, Justo and Torrejon, Alberto},
    pages = {Article 106746}
}

@article{church_review_2022,
	title = {Review of obnoxious facilities location problems},
	volume = {138},
    doi = {10.1016/j.cor.2021.105468},
	journal = {Computers \& Operations Research},
	author = {Church, Richard L. and Drezner, Zvi},
	month = feb,
	year = {2022},
    pages = {Article 105468}
}

@article{bigler_matheuristic_2024,
	title = {A matheuristic for locating obnoxious facilities},
	volume = {166},
    doi = {10.1016/j.cor.2024.106602},
	journal = {Computers \& Operations Research},
	author = {Bigler, Tamara},
	month = jun,
	year = {2024},
    pages = {Article 106602}
}

@article{kalczynski_obnoxious_2026,
	title = {The obnoxious facilities $p$-center problem with forbidden regions},
	volume = {185},
    doi = {10.1016/j.cor.2025.107297},
	journal = {Computers \& Operations Research},
	author = {Kalczynski, Pawel and Drezner, Zvi},
	month = jan,
	year = {2026},
    pages = {Article 107297}
}

@article{chiang_obnox_2017,
	title = {Compact {Model} for the {Obnoxious} \textit{p}-{Median} {Problem}},
	volume = {7},
	doi = {10.4236/ajor.2017.76026},
	number = {6},
	journal = {American Journal of Operations Research},
	publisher = {Scientific Research Publishing},
	author = {Chiang, Yen-I. and Lin, Chang-Chun},
	month = nov,
	year = {2017},
	pages = {348--355}
}

@article{colmenar_obnox_grasp_2016,
	title = {Advanced {Greedy} {Randomized} {Adaptive} {Search} {Procedure} for the {Obnoxious} \textit{p}-{Median} problem},
	volume = {252},
	doi = {10.1016/j.ejor.2016.01.047},
	number = {2},
	journal = {European Journal of Operational Research},
	author = {Colmenar, J. Manuel and Greistorfer, Peter and Martí, Rafael and Duarte, Abraham},
	month = jul,
	year = {2016},
	pages = {432--442}
}

@article{herran_obnox_vns_2020,
	title = {A parallel variable neighborhood search approach for the obnoxious \textit{p}-median problem},
	volume = {27},
	doi = {10.1111/itor.12510},
	number = {1},
	journal = {International Transactions in Operational Research},
	author = {Herrán, Alberto and Colmenar, José M. and Martí, Rafael and Duarte, Abraham},
	year = {2020},
	pages = {336--360}
}

@incollection{Hoffman1963SimpleLP,
    author    = {Hoffman, Alan Jerome},
    title     = {On {S}imple {L}inear {P}rogramming {P}roblems},
    booktitle = {Proceedings of Symposia in Pure Mathematics},
    volume    = {7},
    pages     = {317--327},
    year      = {1963},
    publisher = {American Mathematical Society}
}

@article{BurkardKlinzRudolf1996,
    author    = {Burkard, Rainer E. and Klinz, Bettina and Rudolf, R{\"u}diger},
    title     = {Perspectives of {M}onge properties in optimization},
    journal   = {Discrete Applied Mathematics},
    volume    = {70},
    number    = {2},
    pages     = {95--161},
    year      = {1996},
    publisher = {Elsevier},
    doi       = {10.1016/0166-218X(95)00103-X}
}

@article{queyranne_general_1998,
	title = {A {General} {Class} of {Greedily} {Solvable} {Linear} {Programs}},
	volume = {23},
    number = {4},
    doi = {10.1287/moor.23.4.892},
	journal = {Mathematics of Operations Research},
	author = {Queyranne, Maurice and Spieksma, Frits and Tardella, Fabio},
	month = nov,
	year = {1998},
	pages = {892--908}
}

@article{zhang_review_2025,
	title = {A {Review} of {Research} on the {Transportation} {Problem}},
	volume = {15},
    number = {5},
    doi = {10.4236/ojapps.2025.155081},
	journal = {Open Journal of Applied Sciences},
	author = {Zhang, Nanlan and Xie, Fanrong},
	year = {2025},
	pages = {1168--1177}
}

@book{yemelichev1984polytopes,
    author    = {Yemelichev, V. A. and Kovalev, M. M. and Kravtsov, M. K.},
    title     = {Polytopes, Graphs and Optimisation},
    publisher = {Cambridge University Press},
    year      = {1984}
}

@article{kacher_comprehensive_2021,
	title = {A {Comprehensive} {Literature} {Review} on {Transportation} {Problems}},
	volume = {7},
	doi = {10.1007/s40819-021-01134-y},
	number = {5},
	journal = {International Journal of Applied and Computational Mathematics},
	author = {Kacher, Yadvendra and Singh, Pitam},
	month = sep,
	year = {2021},
    pages = {Article 206}
}

@book{monge1781memoire,
    title={M{\'e}moire sur la th{\'e}orie des d{\'e}blais et des remblais [Memoir on the theory of cut and fill]},
    author={Monge, Gaspard},
    language={fr},
    year={1781},
    publisher={Imprimerie royale}
}

@article{hitchcock_distribution_1941,
	title = {The {Distribution} of a {Product} from {Several} {Sources} to {Numerous} {Localities}},
	volume = {20},
	doi = {10.1002/sapm1941201224},
	number = {1-4},
	journal = {Journal of Mathematics and Physics},
	author = {Hitchcock, Frank L.},
	month = apr,
	year = {1941},
	pages = {224--230}
}

@article{koopmans_optimum_1949,
	title = {Optimum {Utilization} of the {Transportation} {System}},
	volume = {17},
    doi = {10.2307/1907301},
	journal = {Econometrica},
	author = {Koopmans, Tjalling C.},
	year = {1949},
    publisher = {Wiley, Econometric Society},
	pages = {136--146}
}

@article{ford_solving_1956,
	title = {Solving the {Transportation} {Problem}},
	volume = {3},
	doi = {10.1287/mnsc.3.1.24},
	number = {1},
	journal = {Management Science},
	author = {Ford, L. R. and Fulkerson, D. R.},
	month = oct,
	year = {1956},
	pages = {24--32}
}

@incollection{dantzig1951transportation,
    author       = {Dantzig, G. B.},
    title        = {Application of the {S}implex {M}ethod to a {T}ransportation {P}roblem},
    booktitle    = {Activity Analysis of Production and Allocation},
    editor       = {Koopmans, T. C.},
    publisher    = {John Wiley \& Sons},
    address      = {New York},
    year         = {1951},
    pages        = {359--373},
    note         = {Cowles Commission Monograph No.\ 13}
}

@article{dahl_transportation_2008,
	title = {Transportation matrices with staircase patterns and majorization},
	volume = {429},
	doi = {10.1016/j.laa.2008.05.019},
	number = {7},
	journal = {Linear Algebra and its Applications},
	author = {Dahl, Geir},
	month = oct,
	year = {2008},
	pages = {1840--1850}
}

@article{puerto_revisiting_2017,
	title = {Revisiting \textit{k}-sum optimization},
	volume = {165},
	issn = {1436-4646},
	doi = {10.1007/s10107-016-1096-1},
	number = {2},
	journal = {Mathematical Programming},
	author = {Puerto, Justo and Rodríguez-Chía, Antonio Manuel and Tamir, Arie},
	month = oct,
	year = {2017},
	pages = {579--604}
}

@article{PuertoTamir2005,
    author = {Puerto, Justo and Tamir, Arie},
    title = {Locating tree-shaped facilities using the ordered median objective},
    year = {2005},
    publisher = {Springer-Verlag},
    address = {Berlin, Heidelberg},
    volume = {102},
    number = {2},
    doi = {10.1007/s10107-004-0547-2},
    journal = {Mathematical Programming},
    month = mar,
    pages = {313--338}
}

@book{inequalities,
    title={Inequalities},
    author={Hardy, G. H. and Littlewood, J. E. and P{\'o}lya, G.},
    year={1934},
    publisher={Cambridge University Press}
}

@book{bazaraa2010linear,
    title={Linear {P}rogramming and {N}etwork {F}lows},
    author={Bazaraa, Mokhtar S. and Jarvis, John J. and Sherali, Hanif D.},
    year={2010},
    publisher={John Wiley \& Sons},
    edition={4th}
}

@book{stanley1997enumerative,
    title={Enumerative Combinatorics: Volume 1},
    author={Stanley, R. P.},
    year={1997},
    publisher={Cambridge University Press},
    edition={2nd}
}

@misc{gurobi,
    author = {{Gurobi Optimization, LLC}},
    title = {{Gurobi Optimizer Reference Manual}},
    year = 2026,
    note = {Retrieved from https://www.gurobi.com. Accessed January 11, 2026}
}

@article{hillbrecht2025bilevel,
    title={Bilevel {O}ptimization of the {K}antorovich {P}roblem and its {Q}uadratic {R}egularization {P}art {I}{I}{I}: {T}he {F}inite-{D}imensional {C}ase}, 
    author={Hillbrecht, Sebastian},
    journal={arXiv preprint},
    year={2025},
    doi={10.48550/arXiv.2406.08992}
}

@article{singh_bilevel_TP,
    title = {Bilevel time minimizing transportation problem},
    journal = {Discrete Optimization},
    volume = {5},
    number = {4},
    pages = {714--723},
    year = {2008},
    doi = {10.1016/j.disopt.2008.04.004},
    author = {Singh, Sonia and Khandelwal, Ankit and Puri, M. C.}
}

@misc{instances2026dataset,
    author       = {Nickel, Stefan and Puerto, Justo and Ramoser, Simon and Torrejon, Alberto},
    title        = {Instance and solution data for '{R}evisiting transportation problems under {M}onge costs with applications to the {D}iscrete {O}rdered {M}edian {P}roblem'},
    year         = {2026},
    howpublished = {Mendeley Data, V2},
    doi          = {10.17632/kk8ms9rsby.2}
}

@article{AGGARWAL1995,
    title = {Efficient {Minimum} {Cost} {Matching} and {Transportation} {Using} the {Quadrangle} {Inequality}},
    journal = {Journal of Algorithms},
    volume = {19},
    number = {1},
    pages = {116--143},
    year = {1995},
    doi = {10.1006/jagm.1995.1030},
    author = {A. Aggarwal and A. Barnoy and S. Khuller and D. Kravets and B. Schieber}
}

@article{WERMAN1985,
    title = {A distance metric for multidimensional histograms},
    journal = {Computer Vision, Graphics, and Image Processing},
    volume = {32},
    number = {3},
    pages = {328--336},
    year = {1985},
    doi = {10.1016/0734-189X(85)90055-6},
    author = {Michael Werman and Shmuel Peleg and Azriel Rosenfeld}
}

@INPROCEEDINGS{RUBNER1998,
    author={Rubner, Y. and Tomasi, C. and Guibas, L.J.},
    booktitle={Sixth International Conference on Computer Vision (IEEE Cat. No.98CH36271)}, 
    title={A {Metric} for {Distributions} with {Applications} to {Image} {Databases}}, 
    year={1998},
    pages={59--66},
    doi={10.1109/ICCV.1998.710701}
}

@article{TSEYTLIN2023,
    title = {Inventory allocation with full downward substitution and monotone cost differences},
    journal = {European Journal of Operational Research},
    volume = {307},
    number = {1},
    pages = {130--139},
    year = {2023},
    doi = {10.1016/j.ejor.2022.08.027},
    author = {Yulia Tseytlin and H. Sebastian Heese}
}

@article{TSEYTLIN2019,
    title = {Allocation under a general substitution structure},
    journal = {European Journal of Operational Research},
    volume = {277},
    number = {2},
    pages = {492--506},
    year = {2019},
    doi = {10.1016/j.ejor.2019.02.049},
    author = {Yulia Tseytlin and H. Sebastian Heese}
}

@article{burkard_monge_2007,
	title = {Monge properties, discrete convexity and applications},
	volume = {176},
	doi = {10.1016/j.ejor.2005.04.050},
	number = {1},
	journal = {European Journal of Operational Research},
	author = {Burkard, Rainer E.},
	year = {2007},
	pages = {1--14}
}

@article{CELA2018,
    title = {New special cases of the {Quadratic} {Assignment} {Problem} with diagonally structured coefficient matrices},
    journal = {European Journal of Operational Research},
    volume = {267},
    number = {3},
    pages = {818--834},
    year = {2018},
    doi = {10.1016/j.ejor.2017.12.024},
    author = {Eranda \c{C}ela and Vladimir Deineko and Gerhard J. Woeginger}
}

@article{PUERTO2011,
    title = {Single-allocation ordered median hub location problems},
    journal = {Computers \& Operations Research},
    volume = {38},
    number = {2},
    pages = {559--570},
    year = {2011},
    doi = {10.1016/j.cor.2010.07.018},
    author = {J. Puerto and A.B. Ramos and A.M. Rodríguez-Chía}
}

@article{Kantorovich1942,
    author  = {Kantorovich, L. V.},
    title   = {On the translocation of masses},
    journal = {Doklady Akademii Nauk SSSR},
    volume  = {37},
    number  = {7--8},
    pages   = {199--201},
    year    = {1942}
}

@article{Beasley_original_dataset,
    title = {A note on solving large \textit{p}-median problems},
    journal = {European Journal of Operational Research},
    volume = {21},
    number = {2},
    pages = {270--273},
    year = {1985},
    doi = {10.1016/0377-2217(85)90040-2},
    author = {J. E. Beasley}
}

@article{revelle_facility_1993,
	title = {Facility siting and integer-friendly programming},
	volume = {65},
	doi = {10.1016/0377-2217(93)90329-L},
	number = {2},
	journal = {European Journal of Operational Research},
	author = {ReVelle, Charles},
	year = {1993},
	pages = {147--158}
}

\end{document}